\documentclass{amsart}
\usepackage{latexsym}
\usepackage[usenames, dvipsnames]{color}          
\usepackage{amsfonts}                
\usepackage{amsmath}    
\numberwithin{equation}{section}

\usepackage{varioref}
\usepackage[T1]{fontenc}
\usepackage{amsthm}         
\usepackage{amssymb}        

\usepackage[mathcal]{euscript} 
\usepackage{mathrsfs}

\newtheorem{thm}{Theorem}[section]
\newtheorem{cor}[thm]{Corollary}
\newtheorem{lem}[thm]{Lemma}
\newtheorem{prop}[thm]{Proposition}

\theoremstyle{definition}
\newtheorem{rem}[thm]{Remark}
\newtheorem{ex}[thm]{Example}
\newtheorem{defi}[thm]{Definition}

\newcommand{\R}{\mathbb{R}}
\newcommand{\RN}{\mathbb{R}^N}
\newcommand{\G}{\mathbb{G}}

\renewcommand{\d}{\mathrm{d}}
\newcommand{\de}{\partial}
\newcommand{\LL}{\mathcal{L}}
\newcommand{\Lop}{\mathcal{L}}

\newcommand{\Lloc}[1]{L^1_{\mathrm{loc}}(#1)}

\newcommand{\dela}{\delta_{\lambda}}
\newcommand{\longto}{\longrightarrow}

\numberwithin{equation}{section}

\title[Fundamental solution for homogeneous H\"ormander operators]
 {The existence of a global fundamental solution\\ for homogeneous H\"ormander operators\\
  via a global Lifting method}
 \author{Stefano Biagi}
 \address{Dipartimento di Matematica,
         Universit\`{a} degli Studi di Bologna\\
         Piazza di Porta San Donato, 5 - 40126 Bologna, Italy\\
         Fax.: +39-51-2094490.}
 \email{stefano.biagi3@unibo.it}
 \author{Andrea Bonfiglioli}
 \address{Dipartimento di Matematica,
         Universit\`{a} degli Studi di Bologna\\
         Piazza di Porta San Donato, 5 - 40126 Bologna, Italy\\
         Tel.: +39-51-2094498, Fax.: +39-51-2094490.}
 \email{andrea.bonfiglioli6@unibo.it}

\begin{document}
\begin{abstract}
 We prove the existence of a global fundamental solution $\Gamma(x;y)$ (with pole $x$)
     for any H\"ormander operator $\LL=\sum_{i=1}^m X_i^2$ on $\R^n$ which is $\dela$-homogeneous of degree $2$. Here homogeneity is meant with respect to a
 family of non-isotropic diagonal maps $\dela$ of the form $\dela(x)=(\lambda^{\sigma_1}x_1,\ldots,\lambda^{\sigma_n}x_n)$, with
 $1=\sigma_1\leq\cdots\leq\sigma_n$. Due to a global lifting method for homogeneous operators proved by Folland in [\emph{On the Rothschild-Stein lifting theorem},
 Comm. PDEs, 1977], there exists a Carnot group $\G$ and a polynomial surjective map $\pi:\G\to \R^n$ such that
 $\LL$ is $\pi$-related to a sub-Laplacian $\LL_\G $ on $\G$. We show that it is always possible to perform a (global) change of variable
 on $\G$ such that the lifting map $\pi$ becomes the projection
 of $\G\equiv \R^n\times \R^p$ onto $\R^n$. We prove that an integration argument over the (non-compact) fibers of $\pi$ provides
 a fundamental solution for $\LL$. Indeed, if $\Gamma_\G(x,{x}';y,{y}')$ ($x,{x}'\in\R^n$; $y,{y}'\in\R^p$)
 is the fundamental solution of $\LL_\G $, we show that
 $\Gamma_\G(x,0;y,{y}')$ is always integrable w.r.t.\,${y}'\in \R^p$, and its $y'$-integral
 is a fundamental solution for $\LL$.
\end{abstract}
 \maketitle


\section{Introduction} \label{sec.intro}
 Let us consider, on Euclidean space $\R^n$, a second-order
  linear partial differential operator
  (PDO, in the sequel) $\LL$ with smooth coefficients,
  and let $\mathcal{H}=\LL-\de_t$ be
  the associated Heat operator on $\R^{n+1}\equiv \R^n_x\times\R_t$.
  Roughly put, if $\mathcal{H}$ admits a Heat kernel
  $\{p_t(x;y)\}_{t > 0}$ which is integrable on $(0,\infty)$ w.r.t.\,$t$, then
  in many meaningful cases  the
  function
  \begin{equation}\label{saturintro}
   \Gamma: \{(x,y)\in\R^n\times \R^n: x \neq y\} \longrightarrow \R, \quad
   \Gamma(x;y) = \int_0^\infty p_t(x;y)\,\d t,
  \end{equation}
  is a fundamental solution for the operator $\LL$ on $\R^n$.
  To fix the naming, we say that $\Gamma(x;y)$ in \eqref{saturintro} is obtained from $p_t(x;y)$ via the \emph{saturation}
  of the variable $t$; since, for any function $u(x,t)$ on $\R^{n+1}$ which is independent of $t$, $\mathcal{H}u(x,t)$
  coincides with $\LL u(x)$ for any $x\in\R^n$ and any $t\in \R$, we say that the operator $\mathcal{H}$ is a \emph{lifting} of $\LL$.

  The same lifting/saturation process also occurs in another classical case, which -for its simplicity- is
  worth mentioning: if $\Delta_n=\sum_{j=1}^n (\de_{x_j})^2$ is the usual Laplace operator
  on $\R^n$, then $\Delta_{n+p}$ is a lifting of  $\Delta_n$ whenever $p\geq 1$. If $n>2$, the fundamental solution
  of $\Delta_n$ (a constant multiple of $\|x\|^{2-n}$) can be recovered by that of $\Delta_N$ by saturation:
  this is easily seen by the integral identity
  $${c}\,{\Big(\sqrt{x_1^2+\cdots+x_n^2}\Big)^{2-n}}=
  \int_{\R^p}{\Big(\sqrt{x_1^2+\cdots+x_n^2+t_1^2+\cdots+t_p^2}\Big)^{2-n-p}}\,\d t_1\ldots \d t_p, $$
  holding true for some constant $c$ whenever $n>2$, the identity following
  from the change of variable $t=\|x\|\,\tau$ (with $\tau\in \R^p$).
  Incidentally,  the above integral is convergent if $n>2$.\medskip

  In the literature, there are many other examples of lifting involving meaningful PDOs: for instance, consider the case of the
  following Grushin operator
\begin{equation}\label{intro.grushin}
 \text{$\mathcal{G}=(\de_{x_1})^2+(x_1\,\de_{x_2})^2$\quad on $\R^2$,}
\end{equation}
 a lifting of which is given by the PDO
  $$\widetilde{\mathcal{G}}=(\de_{x_1})^2+(\de_{x_3}+x_1\,\de_{x_2})^2,\qquad (x_1,x_2,x_3)\in \R^3.$$
  In its turn, the latter is nothing but a copy (via a change of variable)
  of the well known Kohn-Laplacian on the first Heisenberg group.
  Parallel to what happens in the mentioned case of $\LL$ and its heat lifting $\mathcal{H}$ or in the case of
  $\Delta_n$ and $\Delta_{n+p}$,
  the idea of obtaining a fundamental solution for the Grushin operator $\mathcal{G}$ via a saturation argument
  applied to the (explicit!) fundamental solution of $\widetilde{\mathcal{G}}$ has already been
  exploited in the literature (e.g., in Bauer, Furutani, Iwasaki \cite{BauerFurutaniIwasaki};
  see also Calin, Chang, Furutani, Iwasaki \cite[Sect. 10.3]{CalinChangFurutaniIwasaki} for the Heat kernel;
  more generally,
  see Beals, Gaveau, Greiner, Kannai \cite{BealsGaveauGreinerKannai2} for operators lifting to
  sub-Laplacians on $2$-step Carnot groups).\medskip

  The main aim of this paper is to obtain a process of lifting/saturation in a
  quite general setting: for instance, we shall show that
  any H\"ormander operator $\LL=\sum_{i=1}^m X_i^2$ on $\R^n$, which is
  the sum of squares of vector fields $X_1,\ldots,X_m$
  $\dela$-homogeneous of degree $1$,
  admits a lifting which is a sub-Laplacian $\LL_\G $ on some Carnot group $\G$. Here homogeneity is meant with respect to a
 family of non-isotropic diagonal maps $\dela$ of the form
\begin{equation}\label{intro.dela}
 \dela(x)=(\lambda^{\sigma_1}x_1,\ldots,\lambda^{\sigma_n}x_n),\quad
 \text{with
 $1=\sigma_1\leq\cdots\leq\sigma_n$.}
\end{equation}
 Most importantly, since $\LL_\G $ admits a global
 fundamental solution $\Gamma_\G$ (see Folland \cite{Folland75}), we shall prove that a global
 fundamental solution for $\LL$ \emph{can always be obtained} by a saturation argument from  $\Gamma_\G$
 (provided that $\sum_{i=1}^n\sigma_i$, the so-called homogeneous dimension of $(\R^n,\dela)$, is $>2$).
  This gives an \emph{existence} result for a fundamental solution $\Gamma$ for $\LL$, together with an integral representation
  for $\Gamma$ in terms of $\Gamma_\G$. In selected cases (namely when $\Gamma_\G$ is explicitly known, possibly under an integral form)
  one also obtains a rather explicit (integral) representation for $\Gamma$.
  This is possible also because $\G$ can be constructed in a quite simple way.\medskip

 We describe more closely our procedure.
 The Lifting method introduced by
 Rothschild and Stein \cite{RothschildStein}
 is a remarkable tool for obtaining a lifting $\widetilde{\LL}$ of any
 H\"ormander operator $\LL$, with $\widetilde{\LL}$ \emph{locally} approximated
 by a sub-Laplacian on a (free) Carnot group. Due to our purpose to obtaining a \emph{global}
 fundamantal solution, this general result does not seem applicable.
 What is instead perfectly suited to our case is the global Lifting for \emph{homogeneous} operators proved by Folland
 in \cite{Folland2}, which we now recall.
 Let $X_1,\ldots,X_m$ be $C^\infty$ vector fields in $\R^n$, homogeneous of degree $1$ w.r.t.
 $\dela$ in \eqref{intro.dela}, that is
 $$X_i(f\circ\dela)= \lambda\,(X_if)\circ \dela \qquad \forall\,\,\lambda>0,\quad\forall\,\,f\in C^\infty(\R^n,\R),$$
  for any $i=1,\ldots,m$.
 We assume that $X_1,\ldots,X_m$ are linearly independent (as linear differential operators)
 and they satisfy H\"ormander's hypoellipticity  condition.
\begin{ex}\label{examp.Grushin}
  For example, the vector fields $X_1=\de_{x_1}$ and $X_2=x_1\de_{x_2}$ associated with the Grushin operator
 $\mathcal{G}$ on $\R^2 $ in \eqref{intro.grushin}
 fulfil these assumptions\footnote{The linear dependence of the vectors $X_1(x)\equiv (1,0)$ and $X_2(x)\equiv (0,x_1)$
 when $x_1=0$ is not in contrast with their linear independence as PDOs in the linear space of the smooth vector fields on $\R^2$.}
 with the family of dilations $\dela(x_1,x_2)=(\lambda^{1} x_1,\lambda^{2} x_2)$.
 The same is true of $X_1=\de_{x_1}$, $X_2=x^2_1\de_{x_2}$ on $\R^2$ (modeling a higher-step Grushin operator) with dilations
  $\dela(x_1,x_2)=(\lambda^{1} x_1,\lambda^{3} x_2)$.
\end{ex}
 Let $\mathfrak{a}=\textrm{Lie}\{X_1,\ldots,X_m\}$ be the Lie algebra generated by $X_1,\ldots,X_m$, i.e.,
 the smallest Lie subalgebra of the smooth vector fields on $\R^n$ containing $X_1,\ldots,X_m$.
 Then, by the $\dela$-homogeneity of the $X_i$'s it is easy to recognize that $\mathfrak{a}$ is  {nilpotent} and  {stratified}, that is
 $$\mathfrak{a}=\mathfrak{a}_1\oplus \mathfrak{a}_2\oplus\cdots \oplus \mathfrak{a}_r,\qquad\text{and}\quad
   \begin{cases}
   [\mathfrak{a}_1,\mathfrak{a}_{i-1}]=\mathfrak{a}_i & \hbox{$2\leq i\leq r$} \\
   [\mathfrak{a}_1,\mathfrak{a}_{r}]=\{0\}, &
   \end{cases}$$
 where $\mathfrak{a}_1=\textrm{span}\{X_1,\ldots,X_m\}$.
 We equip $\mathfrak{a}$ with the  {Campbell-Baker-Hausdorff} operation
\begin{equation}\label{CBHoper}
  X\diamond Y=X+Y+\frac{1}{2}[X,Y]+\frac{1}{12}[X[X,Y]]-\frac{1}{12}[Y[X,Y]]+\cdots,
\end{equation}
 the sum being finite due to the nilpotency  of $\mathfrak{a}$.
 We can choose a linear basis $\mathcal{A}$ of $\mathfrak{a}$ by grouping
 together some fixed bases of  $\mathfrak{a}_1,\mathfrak{a}_2,\ldots,\mathfrak{a}_r$
 and such that $X_1,\ldots,X_m$ belong to $\mathcal{A}$.
 Setting $N:=\textrm{dim}(\mathfrak{a})$, we
 identify $\mathfrak{a}$ with $\R^N$ via coordinates w.r.t.\,$\mathcal{A}$.

 Then $\mathbb{A}=(\mathfrak{a},\diamond)$ is a {Carnot} group
 (with underlying manifold $\mathfrak{a}\equiv \R^N$) with Lie algebra $\textrm{Lie}(\mathbb{A})$ isomorphic to the Lie algebra $\mathfrak{a}$.
 But $\mathfrak{a}$ is not only a Lie algebra (and a Carnot group, when equipped with the operation $\diamond$): it is a Lie algebra \emph{of vector
 fields on $\R^n$}, so it can be ``exponentiated'', that is one can consider the integral curves in $\R^n$ of any $X\in \mathfrak{a}$.
 Due to $\dela$-homogeneity reasons, any such $X$ is global, i.e., its integral curves are defined for any real time.
 Hence the following map is well defined
 \begin{equation*}
    \pi:\mathfrak{a}\longrightarrow \R^n,\quad \pi(X):=\Big(\Psi_t^X(0)\Big)\Big|_{t=1},
 \end{equation*}
 where $t\mapsto \Psi^X_t(0)$ denotes the integral curve of $X$ at time $t$ starting from $0\in\R^n$ at $t=0$.

 Let us denote by $J_1,\ldots,J_N$
 the basis of the left-invariant vector fields on $\mathbb{A}$ coinciding with
 the coordinate\footnote{We recall that we fixed on $\mathfrak{a}\equiv \RN$
 the coordinates w.r.t.\,a fixed basis $\mathcal{A}$ adapted to the stratification of $\mathfrak{a}$.}
 partial derivatives at $0\in\RN\equiv \mathfrak{a}$. It is not difficult to recognize that, thanks to the H\"ormander
 condition satisfied by $X_1,\ldots,X_m$, the vector fields $J_1,\ldots,J_m$ form a basis of Lie-generators
 for $\textrm{Lie}(\mathbb{A})$, so that  $J_1^2+\cdots+J_m^2$
 is a sub-Laplacian on the Carnot group $\mathbb{A}$ (in the sense of \cite[Section 1.5]{BLU}).

 What Folland proved in \cite{Folland2} are the following remarkable facts: $\pi$ is a {surjective} map; it is
 polynomial;\footnote{See the previous footnote.} and, for $i=1,\ldots,m$, $X_i$ is $\pi$-related to $J_i$, that is:
\begin{equation}\label{intro.lift}
 \d \pi(J_i)_a=(X_i)_{\pi(a)} \qquad \forall\,a\in \mathbb{A},\,\,\forall\,\,i=1,\ldots,m.
\end{equation}
 In particular, the sub-Laplacian $J_1^2+\cdots+J_m^2$ is $\pi$-related to $\LL=X_1^2+\cdots+X_m^2$.\medskip

 Our goal is to perform a change of variable on $\mathbb{A}\equiv \RN$ turning it into
 a homogeneous Carnot group $\G=(\RN,\star,d_\lambda)$ (in the sense of \cite[Section 1.4]{BLU}) and, most
 importantly, turning $\pi$ into a genuine projection map; namely (without changing notation from the old $\pi$ to the
 map $\pi$ read in the new coordinates after the change of variable) we have
 $$\pi:\R^{N}\longrightarrow\R^n\qquad \pi(x_1,\ldots,x_n,x_{n+1},\ldots,x_N)=(x_1,\ldots,x_n).$$
 Our construction of this change of variable is quite technical, and we refer the reader directly to
 Section \ref{sec:changeofv}.
 We further denote by $Z_1,\ldots,Z_m$ the vector fields on $\G$
 obtained from $J_1,\ldots,J_m$ via the change of variable and we set $\LL_\G:=\sum_{j=1}^m Z_j^2$.
 The reason for this change of variable is motivated by the fact that
 we need to control the fibers of $\pi$, since we shall perform an integration over them, and we shall need suitable estimates
 allowing for integration.

 Due to the different roles of the first $n$ variables of $\RN$ if compared to the remaining ones,
 we use, from now on, the following notation for the points of $\G\equiv \RN$:
 $$(x,\xi),\quad \text{with $x=(x_1,\ldots,x_n)\in\R^n$ and $\xi=(\xi_1,\ldots,\xi_p)\in\R^p$,} $$
 where $p=N-n=\textrm{dim}(\textrm{Lie}\{X_1,\ldots,X_m\})-n$.
 With this notation, one can easily check that the dilations $d_\lambda$ on $\RN$ take on the product form
 $d_\lambda(x,\xi)=(\dela(x),\delta^*_\lambda(\xi))$, where
 \begin{equation}\label{deltyastar}
    \delta^*_\lambda(\xi)=(\lambda^{\sigma^*_1}\xi_1,\ldots,\lambda^{\sigma^*_p}\xi_p),\quad
 \text{with
 $1=\sigma^*_1\leq\cdots\leq\sigma^*_p$.}
 \end{equation}
 The fact that $\pi$ is now the projection onto $\R^n$ says that \eqref{intro.lift}
  $Z_i$ is a genuine lifting of $X_i$ (in the sense introduced in the \emph{incipit} of this introduction); this means that,
 for any smooth function $f=f(x,\xi)$ which is independent of $\xi$, one has
 $$Z_i\big(f(x,\xi)\big)= (X_i f)(x),\qquad \forall\,\, x\in \R^n.$$
 Equivalently, there exist vector fields $R_1,\ldots,R_m$ on $\RN$ \emph{operating only in the $\xi$ variables}
 (and with coefficients possibly depending on $(x,\xi)$) such that
 $$Z_i=X_i+R_i,\quad i=1,\ldots,m.$$
 This also proves that $\LL_\G=\sum_{i=1}^m Z_i^2$ is a lifting of $\LL=\sum_{i=1}^m X_i^2$.
\begin{ex}\label{examp.Grushin2}
 In the first case considered in Example \ref{examp.Grushin}, one has
 $X_1=\de_{x_1}$ and $X_2=x_1\de_{x_2}$, and the lifting vector fields are
 $$Z_1=\frac{\de}{\de x_1}=X_1,\quad    Z_2= \frac{\de}{\de \xi_1}+x_1\frac{\de}{\de x_2}=X_2+\frac{\de}{\de \xi_1};$$
 the dilations on $\R^3$ are $d_\lambda(x_1,x_2,\xi)=(\lambda^1 x_1,\lambda^2 x_2,\lambda^1\xi_1)$.

 In the second example, one has
 $X_1=\de_{x_1}$ and $X_2=x^2_1\de_{x_2}$, and the lifting vector fields are
\begin{align*}
 Z_1&=\frac{\de}{\de x_1}-\frac{\xi_1}{2}\,\frac{\de}{\de \xi_2}=X_1
 -\frac{\xi_1}{2}\,\frac{\de}{\de \xi_2},\\
 Z_2&= \frac{\de}{\de \xi_1}+\frac{x_1}{2}\frac{\de}{\de \xi_2}+ x^2_1\frac{\de}{\de x_2}=X_2+\frac{\de}{\de \xi_1}+\frac{x_1}{2}\frac{\de}{\de \xi_2};
\end{align*}
 the dilations on $\R^4$ are $d_\lambda(x_1,x_2,\xi_1,\xi_2)=(\lambda^1 x_1,\lambda^3 x_2,\lambda^1\xi_1,\lambda^2\xi_2)$.
\end{ex}
 We are now ready to state our result proving the existence of a fundamental solution $\Gamma(x;y)$
 of $\LL=\sum_{i=1}^m X_i^2$ obtained by a saturation argument starting from the fundamental solution $\Gamma_\G(x,\xi;y,\eta)$ of the
 sub-Laplacian $\LL_\G $. The formula relating $\Gamma$ to $\Gamma_\G$ is the following one:
\begin{equation*}
  \Gamma(x;y) = \int_{\R^p} \Gamma_\G\big(x,0; y,\eta)\,\d\eta\qquad (x\neq y).
\end{equation*}
 Before stating the complete theorem, we show what it gives in the previous examples.
\begin{ex}\label{examp.Grushin23333}
 In the first case  considered in Example \ref{examp.Grushin},
 the fundamental solution of the operator $\mathcal{G}=(\de_{x_1})^2+(x_1\de_{x_2})^2$
 has the explicit form
\begin{equation*}
  \Gamma(x_1,x_2;y_1,y_2) = c\,\int_{\R}\frac{\d\eta}
  {\sqrt{ ((x_1-y_1)^2+\eta^2 )^2+ 4\,
    (2\,x_2 - 2\,y_2 + \eta\,(x_1+y_1) )^2}},
 \end{equation*}
 for some positive constant $c$; for a more explicit formula (comparable
 to the existing formulas for the Grushin operator), see
 \eqref{sec.th:GammaGrushinExplicit}.
 In the second example, for the fundamental solution of
 $(\de_{x_1})^2+(x^2_1\de_{x_2})^2$
 we obtain the representation
 \begin{equation*}
   \begin{split}
   \Gamma(x_1,x_2;y_1,y_2)&= \int_{\R^2}\Gamma_\G \Big(y_1-x_1,y_2-x_2 + x_1\eta_1(x_1-y_1)-2\,x_1\eta_2,
   \eta_1,\eta_2-\tfrac{1}{2}x_1\eta_1 \Big)\,\d\eta_1\d\eta_2,
   \end{split}
  \end{equation*}
 where $\Gamma_\G$ is the fundamental solution with pole at $0\in\R^4$ of the sub-Laplacian
\begin{align*}
 Z_1^2+Z_2^2=\Big(\frac{\de}{\de x_1}-\frac{\xi_1}{2}\,\frac{\de}{\de \xi_2}\Big)^2+
 \Big(\frac{\de}{\de \xi_1}+\frac{x_1}{2}\frac{\de}{\de \xi_2}+ x^2_1\frac{\de}{\de x_2}\Big)^2,\quad
 (x_1,x_2,\xi_1,\xi_2)\in\R^4.
\end{align*}
 Even if $\Gamma_\G$ is not explicitly known, due to the equivalence of all
 homogeneous norms on a Carnot group, we can infer that
 $\Gamma(x;y)$ is bounded from above and from below (up to
  two structural constants independent of $x,y\in\R^2$) by
  the following integral (here $(x_1,x_2)\neq (y_1,y_2)$)
\begin{equation*}
  \int_{\R^2}
  \bigg\{
    |y_1-x_1 |+ |
   y_2-x_2 + x_1\eta_1(x_1-y_1)
   -2\,x_1\eta_2 |^{1/3}+ |\eta_1 |+
    |\eta_2-\tfrac{1}{2}x_1\eta_1 |^{1/2}
    \bigg\}^{-5}
  \,\d \eta_1\d\eta_2.
\end{equation*}
 This shows how our representation, albeit not explicit,
  can be used to derive estimates of $\Gamma$.
  \end{ex}
 The following theorem is the main result of this paper.
\begin{thm}\label{thm.riassuntivoooo}
 Let $X_1,\ldots,X_m$ be $C^\infty$ vector fields in $\R^n$, homogeneous of degree $1$ w.r.t.
 $\dela$ in \eqref{intro.dela}, linearly independent (as linear differential operators)
 and satisfying H\"ormander's rank condition at $0\in \R^n$. Suppose that
 $q:=\sum_{j=1}^n \sigma_j>2$. Finally let us set $N=\mathrm{dim}(\mathrm{Lie}\{X_1,\ldots,X_m\})$ and $p=N-n$.\medskip

 Then there exists a homogeneous Carnot group  $\G=\big(\RN,\star,d_{\lambda}\big)$
 (with $m$ generators and nilpotent of step $\sigma_n$),
 and there exist vector fields $Z_1,\ldots,Z_m$,  Lie-generators of $\mathrm{Lie}(\G)$, such
 that (for every $i = 1,\ldots,m$) $Z_i$ is a lifting of $X_i$, via the projection $\pi:\R^N=\R^n\times \R^p\to\R^n$ of $\R^N$
 onto the first $n$ variables.\medskip

 Let $\Gamma_\G$ be the (unique) fundamental solution of $\Lop_\G=\sum_{j=1}^m Z_j^2$ with pole at $0\in\RN$
 and vanishing at infinity (whose existence is proved in \cite{Folland75}).
 Then
\begin{equation}\label{stimosaintrp}
  \Gamma(x;y) = \int_{\R^p} \Gamma_\G \big((x,0)^{-1}\star (y,\eta)\big)\,\d\eta\qquad (x\neq y)
\end{equation}
 is a fundamental solution of $\LL$, that is,\footnote{Note that $\LL$ is formally self-adjoint, due to $\dela$-homogeneity.}
     $y\mapsto \Gamma(x;y)$
    is locally integrable on $\R^n$ and
\begin{equation*}
   \int_{\R^n}\Gamma(x;y)\,\LL\varphi(y)\,\d y = -\varphi(x)
   \qquad \text{for every $\varphi\in C_0^{\infty}(\R^n)$ and every $x\in\R^n$.}
\end{equation*}
 Moreover $\Gamma$ is  strictly positive, and it has the (joint) homogeneity property
 \begin{equation*}
    \Gamma (\delta_\lambda(x);\delta_\lambda(y) ) = \lambda^{2-q}\,
    \Gamma(x;y), \quad \text{for every $x,y \in \R^n$ with $x \neq y$ and $\lambda>0$}.
  \end{equation*}
 Furthermore, it is symmetric:
 $$\Gamma(x;y)=\Gamma(y;x) \qquad \text{for every $x,y\in\R^n$ with $x\neq y$.}$$
 Finally, for every $x \in \R^n$, we have the following properties of the functions $\Gamma(x;\cdot)=\Gamma(\cdot;x)$:\medskip
    \begin{itemize}
    \item[(i)]
    $y\mapsto \Gamma(x;y)$
    is smooth and $\LL$-harmonic on $\R^n\setminus\{x\}$;

    \item[(ii)]
    $y\mapsto \Gamma(x;y)$
    vanishes at infinity (uniformly for $x$ in compact sets);

    \item[(iii)]
         $(x,y)\mapsto \Gamma(x;y)$
  is locally integrable on $\R^n\times \R^n$
    and smooth out of the diagonal.\medskip
\end{itemize}
 There exists only one function $\Gamma$ satisfying the above properties. \medskip
\end{thm}
 In order to demonstrate that $y\mapsto \Gamma(x;y)$ solves $-\LL (\Gamma(x;\cdot))=\text{Dir}_x$
 (the latter being the Dirac mass concentrated at $\{x\}$) we shall make use of a general
 result, of independent interest, providing sufficient conditions for a process of saturation to be applicable
 (see Theorem \ref{sec.one:thm_Main}).
 Moreover, due to the equivalence of all $d_\lambda$-homogeneous norms, from
 \eqref{stimosaintrp} one can obtain uniform estimates of $\Gamma$, from above and from below, on $\R^n\times\R^n$.
 More precisely, with the notation  \eqref{intro.dela} for $\dela$
 and \eqref{deltyastar} for $\delta^*_\lambda$,  let us set
\begin{equation*}
   h(x,\xi) =
   \sum_{j = 1}^n|x_j|^{ {1}/{\sigma_j}} + \sum_{k = 1}^p |\xi_k|^{{1}/{\sigma^*_k}},\qquad x\in\R^n,\,\,\xi\in\R^p.
\end{equation*}
 Then there exists a constant $\mathbf{c}$
 (depending only on $\G$)
 such that
   $$\mathbf{c}^{-1}\,\int_{\R^p} h^{2-Q}((x,0)^{-1}\star(y,\eta) )\,\d \eta
   \leq \Gamma(x;y) \leq \mathbf{c}\,
   \int_{\R^p} h^{2-Q}((x,0)^{-1}\star(y,\eta) )\,\d \eta,$$
   holding true for  $x,y \in \R^n$ with $x\neq y$; here we have used the notation $Q=\sum_{j=1}^n \sigma_j+
   \sum_{k=1}^p \sigma^*_k$ for the $d_\lambda$-homogeneous dimension of $\G$.

   We postpone to a future investigation the study of the
   Heat kernel associated with $\LL-\de_t$ (see also \cite{BLUade}) and
   fine estimates of $\Gamma$ in terms of the Carnot-Carath\'eodory
   distance associated with $X_1,\ldots,X_m$.\bigskip

 We end the introduction with some bibliographical references.
 In the literature, when the existence of a \emph{global} fundamental solution $\Gamma$
 for a PDO is provided, it seems that in the vast majority of cases (though exceptions are available):\medskip
 \begin{itemize}
   \item[-] PDOs with \emph{polynomial coefficients} (or having a polynomial growth) are considered;
   \item[-] existence is a \emph{by-product of an explicit formula} (possibly under an integral form)
   for $\Gamma$.\medskip
 \end{itemize}
  Incidentally, the same happens in the present paper, since homogeneous operators
  have necessarily polynomial coefficients, and an integral representation for $\Gamma$ (albeit not explicit)
  is furnished. Global fundamental solutions, without an explicit representation, are given
  e.g., in Folland \cite{Folland75}; Nagel, Ricci, Stein \cite{NagelRicciStein};
  the second-named author and Lanconelli \cite{BonfLancCPAA}. Existence results
  without an exact representation are also available, based on the so-called Levi's parametrix method \cite{Levi}:
  the interested reader is referred to e.g., \cite{BramBranLU,Friedman,HormBOOK};
  concerning Levi's parametrix method,
  we also highlight the recent paper by Bramanti, Brandolini, Manfredini, Pedroni \cite{BramBranMP}, where a local Lifting technique and a local
  saturation argument are also applied. See also the paper \cite{CittiManfredini} by  Citti, Manfredini, where it is exploited
  a local Lifting technique
  involving hypoelliptic H\"ormander operators and their local fundamental solutions.

  Since our method is based on a saturation argument, we would like to highlight that
  a similar saturation method was also used by Beals, Gaveau, Greiner, Kannai in \cite{BealsGaveauGreinerKannai2}, where
  operators lifting to sub-Laplacians on Carnot groups of step two are considered;
  or in Bauer, Furutani, Iwasaki \cite{BauerFurutaniIwasaki} (where the sub-Laplacian of the first Heisenberg group
  is used as a lifting a Grushin operator on $\R^2$).
  Our theorem, allowing for general homogeneous operators, comprises both of these cases.
  We also observe that some general results on fiber integration
  for obtaining Heat kernels appear in
  \cite[Sect. 10.3]{CalinChangFurutaniIwasaki}, where it is required that the
  fibers be compact, but this does not hold in our case.
   To the best of our knowledge, our paper
  is the first systematic analysis of this lifting/saturation technique for all homogeneous
  sums of squares.  \medskip

  The literature on the (difficult) problem of obtaining explicit/integrally-represented fundamental solutions is wide. Starting from early
  works (dating back to 1930-40's) on kinetic operators by Kolmogorov \cite{Kolmogorov} and Chandrasekhar \cite{Chandrasekhar}, the focus
  (in the 1970's) shifted to operators on the Heisenberg groups:
  Folland \cite{Folland}; Folland, Stein \cite{FollandStein};
  Tsutsumi \cite{Tsutsumi} (for some parabolic pseudo-differential operators);
  Hulanicki \cite{Hulanicki}; Gaveau \cite{Gaveau}; Kaplan \cite{Kaplan} (for the case of H-type groups). In \cite{Gaveau, Cygan}
  Heat kernels on nilpotent Lie groups of step two are also considered (on the same topic, see also the more recent paper by Furutani \cite{Furutani}).

  Starting from the paper \cite{Greiner} by Greiner in 1979, a large part of the literature
  has subsequently focused on PDOs of \emph{Grushin-type}, the golden age for this issue
  being the late 1990's with fundamental papers by
  Beals, Gaveau, Greiner \cite{BealsGaveauGreiner1,BealsGaveauGreiner2,BealsGaveauGreiner2ter,BealsGaveauGreiner3,BealsGreinerGaveau},
  and by the same authors and Kannai \cite{BealsGaveauGreinerKannai}; see also Beals \cite{Beals,Beals2}.
  In \cite{BealsGaveauGreiner2} sub-Laplacians on general step two Carnot groups are considered.

  The interest in PDOs on Heisenberg groups (or on groups modeled on the Heisenberg groups, such as quaternionic or Cayley Heisenberg) has not ceased during the decades, see:
  Benson, Dooley, Ratcliff \cite{BensonDooleyRatcliff}; Klingler \cite{Klinger}; Beals, Gaveau, Greiner \cite{BealsGaveauGreiner4};
  Zhu \cite{Zhu}; Luan, Zhu \cite{LuanZhu,LuanZhu2}; Zhu, Yang \cite{ZhuYang}; Boggess, Raich \cite{BoggessRaich}.

  Explicit fundamental solutions are available also for \emph{non-linear} PDOs, such as
  $p$-Laplacians (for Heisenberg/Grushin-type PDOs):
  Capogna, Danielli, Garofalo \cite{CapognaDanielliGarofalo};
  Heinonen, Holopainen \cite{HeinonenHolopainen};
  Bieske, Gong \cite{BieskeGong}; Bieske \cite{Bieske}.

  Recent papers by Agrachev, Boscain, Gauthier, Rossi \cite{AgrachevBoscainGauthier},
  and by Boscain, Gauthier, Rossi \cite{BoscainGauthierRossi} deal with heat kernels, respectively,
  on unimodular Lie groups of type I (in the sense of \cite{Dixmier}, comprising
  the real connected nilpotent Lie groups) and on $3$-step nilpotent Engel/Cartan groups.

  Different examples for which an explicit fundamental solution is available
  can be found in Aar$\tilde{\text{a}}$o \cite{Aarao,Aarao2} (kinetic operators
  of physical interest) and in Calin, Chang \cite{CalinChang}
  (PDOs with radical coefficients).\medskip

 A comprehensive list of references can be found in the 2011 monograph
 by Calin, Chang, Furutani, Iwasaki \cite{CalinChangFurutaniIwasaki},
 where many theories and techniques for obtaining explicit Heat kernels
 for elliptic and sub-elliptic PDOs are presented.\bigskip

  The plan of the paper is the following.
  In Section \ref{sec.one} we
  introduce the relevant definitions and notation used throughout, and we prove
  a general saturation theorem under minimal assumptions (Theorem \ref{sec.one:thm_Main}).
  Afterwards, we focus on H\"ormander sums of squares of homogeneous operators:
  in Section \ref{sec.two}, we first recall Folland's version of the Lifting technique in this framework;
  then in Section \ref{sec:changeofv} we add a change-of-variable  argument to Folland's Lifting.
  In Section \ref{sec:changeofv} we also prove a crucial technical result (Theorem \ref{sec.two_1:thm_regliftLopX})
  showing that the general saturation process of  Section \ref{sec.one} is allowed
  in the homogeneous case. Section \ref{sec:GreenHOM} contains the proof of our main
  Theorem \ref{thm.riassuntivoooo}, as a consequence of the results of the previous sections.
  Finally, in Section \ref{sec:examples} we furnish some explicit examples
  of operators to which our theory  applies.\bigskip

 \emph{Acknowledgments.} The results of this manuscript were presented by the second-named at the
  Conference ``Noncommutative Analysis and Partial Differential Equations'', 11--15 April, 2016,
  Imperial College, London; the second-named author wishes to express
  his gratitude to the Organizing Committee of the Conference for the hospitality.
  \section{A general saturation argument for obtaining fundamental solutions} \label{sec.one}
  To begin with, since there is no common agreement on the notion of what
  fundamental solutions are, we fix the relevant definitions. In what follows we use the notation
   $$D_x^\alpha=\Big(\frac{\de}{\de x}\Big)^\alpha=\frac{\de^{|\alpha|}}{\de x_1^{\alpha_1}\cdots \de x_n^{\alpha_n}},$$
   for higher order derivatives on $\R^n$. Here $\alpha=(\alpha_1,\ldots,\alpha_n)\in (\mathbb{N}\cup\{0\})^n$ and $|\alpha|=\alpha_1+\cdots+\alpha_n$.
\begin{defi}[Fundamental solution]\label{sec.one:defi_Green}
 On Euclidean space $\R^n$, we consider a linear partial differential operator of order $d\in \mathbb{N}$,
\begin{equation}\label{ooperP}
  P = \sum_{|\alpha|\leq d}a_{\alpha}(x)\,D_x^{\alpha},
\end{equation}
   with smooth real valued coefficient functions $a_{\alpha}(x)$ on $\R^n$.
  We say that a function
  $$\Gamma:\{(x;y)\in \R^n\times\R^n: x \neq y\} \longto \R,$$
  is a \emph{(global) fundamental solution for $P$}
  if it satisfies the following
  property:
\begin{enumerate}
  \item[(i)] For every fixed $x \in \R^n$, the function
  $\Gamma(x;\cdot)$ is locally integrable on $\R^n$ and
\begin{equation} \label{sec.two:eq_GreenGamma}
   \int_{\R^n}\Gamma(x;y)\,P^*\varphi(y)\,\d y = -\varphi(x)
   \qquad \text{for every $\varphi\in C_0^{\infty}(\R^n)$},
\end{equation}
  where $P^*$ denotes the usual formal adjoint of $P$.
  \end{enumerate}
\end{defi}
 \noindent In the literature, many authors ask for other properties in order to define a
 fundamental solution for $P$; some of these further requirements are listed below; we explicitly mention them
 since we will be able to prove that our operators $\LL=\sum_jX_j^2$ fulfil many of them as well:\medskip
\begin{enumerate}
  \item[(ii)] $\Gamma(x;y)\geq 0$ whenever $x\neq y$ (or $\Gamma(x;y)> 0$);

  \item[(iii)]
  $\Gamma\in L^1_{\mathrm{loc}}(\R^n\times\R^n)$ and,
  for every fixed $y \in \R^n$, $\Gamma(\cdot;y)$
  is locally integrable on $\R^n$;

  \item[(iv)]
  for every fixed $x \in \R^n$, the function $y\mapsto \Gamma(x;y)$
  vanishes as $y\to \infty$.

  \item[(v)]
  for every fixed $x \in \R^n$, the function $y\mapsto \Gamma(x;y)$
  goes to $\infty$ as $y\to x$.\medskip
\end{enumerate}

 If $\Gamma$ is a fundamental solution for $P$
 and if $x \in \R^n$ is fixed, then
  \eqref{sec.two:eq_GreenGamma} can be rewritten as
\begin{equation} \label{sec.two:eq_GreenGammadist}
   P\Gamma_x = -\mathrm{Dir}_x \quad \text{in $\mathcal{D}'(\R^n)$},
\end{equation}
  where $\mathrm{Dir}_x$ is the Dirac distribution supported at $\{x\}$.
  We give some remarks in order to describe the meaning of the
  above extra requirements (ii)-to-(v).
\begin{rem}
  (a)\,\, The existence of a global fundamental solution
  for $P$ is far from being obvious and it is, in general,
  a very delicate issue.
  In the particular case of $C^{\infty}$-hypoelliptic linear PDOs $P$
  having a $C^{\infty}$-hypoelliptic formal adjoint $P^*$,
  it is possible to prove the \emph{local}
  existence of a fundamental solution on a suitable neighborhood of
  each point of $\R^n$ (see, e.g., \cite{Treves});
  moreover, in \cite{Bony}
  Bony showed that any H\"ormander operator
  admits a smooth fundamental solution on every bounded open set
  satisfying suitable regularity assumptions on the boundary.\medskip

  (b)\,\, Fundamental solutions are, in general, not unique
  since the addition of a $P$-harmonic function (that is, a smooth function
  $h$ such that $P h=0$ in $\R^n$) to a fundamental solution produces another
  fundamental solution. \medskip

  (c)\,\, Nonetheless, if  $P$ is a second order $C^\infty$-hypoelliptic operator with $P(1)\leq 0$,
  and fulfilling
  the Weak Maximum Principle on every bounded open set of $\R^n$, then
  there exists at most one $\Gamma$ satisfying properties (i) and (iv)
  above. Indeed, if $\Gamma_1,\Gamma_2$ are two such
  functions, then
  (for every fixed $x \in \R^n$)
  $u_x := \Gamma_1(x,\cdot)-\Gamma_2(x,\cdot)$
  belongs to $\Lloc{\R^n}$
  and it is a solution of $P u_x = 0$ in the weak sense of distributions on $\R^n$;
  the hypoellipticity of $P$ ensures that $u_x$ is (a.e. equal to) a smooth function on $\R^n$
  which vanishes at infinity by properties (iv) of $\Gamma_1,\Gamma_2$; from the Weak Maximum Principle for $P$
  (and $P(1)\leq 0$)
  it is standardly obtained that $u_x\equiv 0$ (a.e.), i.e., $\Gamma_1\equiv \Gamma_2$ (a.e.).
  When continuity of $\Gamma(x;\cdot)$ is also requested, this gives $\Gamma_1\equiv \Gamma_2$.\medskip

  (d) If $\Gamma$ is a fundamental solution for $P$ satisfying condition (iii)
  above, then it is almost tautological to verify that,
  for any $\varphi\in C_0^\infty(\R^n)$,
  the function
  $$y\mapsto \Lambda(\varphi)(y):=\int_{\R^n}\Gamma(x;y)\,\varphi(x)\,\d x$$
  is locally integrable and it satisfies $P(\Lambda(\varphi))=-\varphi$
  in the weak sense of distributions. Therefore, if $P$ is $C^\infty$-hypoelliptic and whenever
  $\Gamma$ satisfies some integrability conditions ensuring that $\Lambda(\varphi)(y)$ is continuous w.r.t.\,$y$, then
  $-\Lambda(\varphi)$ is a smooth classical solution to $Pu=\varphi$. This is one of the alternative  definitions
  of fundamental solution.\medskip

  (e) Suppose that conditions (iv) and (v) and the \emph{strict positivity} of $\Gamma$ hold true.
 Then the so-called $\Gamma$-balls
  $$\Omega_r(x):=\{y\in \R^n:\,\Gamma(x;y)>1/r\}\cup\{x\} $$
  form a basis of neighborhoods of $x$ (thanks to condition (v)) invading $\R^n$, i.e.,
  $\bigcup_{r>0} \Omega_r(x)=\R^n$ (thanks to condition (iv) and the positivity of $\Gamma$).
  Applications of these further geometrical assumptions to the Potential Theory for a class
  of second-order operators $\LL$ in divergence-form have been given in the series of papers \cite{AbbondanzaBonf,BattagliaBonf,BonfLancJEMS,BonfLancTomm}.
  This shows that the qualitative well-behaved properties of a fundamental solution $\Gamma$ can produce
  meaningful properties of the sheaf of the harmonic and sub-harmonic functions for $\LL$.
\end{rem}
   Let $P$ be a smooth linear PDO on $\R^n$ as in \eqref{ooperP}.
   We say that a linear PDO $\widetilde{P}$, defined on a higher-dimensional
   space $\R^n\times\R^p$, is a
   \emph{lifting of $P$}
   if the following conditions are fulfilled:
\begin{enumerate}
    \item[(a)] $\widetilde{P}$ has smooth coefficients, possibly
    depending on $x\in\R^n$ and $\xi\in\R^p$;
    \item[(b)] for every fixed $f \in C^\infty(\R^n)$, one has
\begin{equation} \label{sec.one:def_eqLlift}
    \widetilde{P}(f\circ \pi)(x,\xi) = (P f)(x),
    \quad \text{for every
    $(x,\xi) \in \R^n\times\R^p$},
\end{equation}
   where $\pi(x,\xi)=x$ is the canonical projection of $\R^n\times\R^p$ onto $\R^n$.
   \end{enumerate}
  It is immediate to recognize that \eqref{sec.one:def_eqLlift} holds true
  if and only if
\begin{equation} \label{sec.one:rem_eqformP}
 \widetilde{P}=P+R\quad \text{with}\qquad  R = \sum_{\beta \neq 0}
   r_{\alpha,\beta}(x,\xi)\,D^\alpha_x D^\beta_{\xi},
\end{equation}
 for (finitely many) coefficient functions $r_{\alpha,\beta}\in C^\infty(\R^N)$,
 possibly identically vanishing on $\R^N$.     In other words,
 every summand of $R$ operates,  at least once necessarily, in the $\xi_1,\ldots,\xi_p$ variables.
 Our use of the term `lifting' is more specific than what is usually done in
 Differential Geometry;\footnote{It is sometimes customary to say that a smooth map $\pi:M\to N$ is
 a lifting of $P$ to $\widetilde{P}$ if
 \eqref{sec.one:def_eqLlift}  is replaced by
 $$\widetilde{P}(f\circ \pi)(m) = (P f)(\pi(m)),\quad \forall\,\,m\in M,\,\,\forall\,\, f\in C^\infty(N).$$}
 throughout, it is understood that we refer to a lifting in the above sense. \medskip

 If $\widetilde{P}$ is a lifting of $P$ and if $\widetilde{P}$ admits a fundamental solution
 $\widetilde{\Gamma}$,
 it is not at all obvious if the same holds true for $P$, nor if a fundamental solution for $P$
 may be obtained via a saturation argument.
 Technically, this is the case if the following heuristic argument can be made rigorous.
 By the definition of fundamental solution for $\widetilde{P}$ we have
 $$-\widetilde{\varphi}(x,\xi)=\int_{\R^n\times \R^p}
 \widetilde{\Gamma}(x,\xi;y,\eta)\,\widetilde{P}^*\widetilde{\varphi}(y,\eta)\,
 \d y \d\eta,$$
 for any $\widetilde{\varphi}\in C_0^\infty(\R^n\times \R^p)$.
 If we take $\widetilde{\varphi}$ of the form $\varphi(x) \theta_j(\xi)$
 (with $\varphi\in C_0^\infty(\R^n)$ and $\theta_j$ in $ C_0^\infty(\R^p)$)
 and we recall that $\widetilde{P}=P+R$, the above equality gives (by choosing $\xi=0$)
\begin{gather}\label{heurist}
\begin{split}
 -\varphi(x)\,\theta_j(0)&=\int_{\R^n}\Big(\int_{\R^p}
 \widetilde{\Gamma}(x,0;y,\eta)\,\theta_j(\eta)\,\d \eta\Big) P^*\varphi(y)\,\d y\\
 &\qquad\qquad +
 \int_{\R^n\times \R^p}
 \widetilde{\Gamma}(x,0;y,\eta)\,R^*\big(\varphi(y)\theta_j(\eta)\big)\,\d y\d \eta=\mathrm{I}_j+\mathrm{II}_j.
\end{split}
\end{gather}
 We want to pass to the limit as $j\to\infty$ in such a way that the above identity produces
 $$-\varphi(x)=\int_{\R^n}\Big(\int_{\R^p}
 \widetilde{\Gamma}(x,0;y,\eta)\,\d \eta\Big) P^*\varphi(y)\,\d y,$$
 so that a fundamental solution for $P$ is available by
 saturating the $\eta$ variable in $\widetilde{\Gamma}(x,0;y,\eta)$.
 Our idea is to choose a sequence $\theta_j\in C_0^\infty(\R^p)$
 such that the set $\{\eta:\theta_j(\eta)=1\}$ invades $\R^p$ as $j\to \infty$,
 and such that $\mathrm{II}_j$ in \eqref{heurist} goes to $0$ as $j\to \infty$.
 This may be reasonably possible (together with some integrability assumptions
 on $\widetilde{\Gamma}$) provided some conditions are fulfilled by the remainder operator $R$:\medskip
 \begin{itemize}
   \item[-]  if one chooses $\theta_j(\eta)=\theta(\eta/j)$
 (for some $\theta\in C_0^\infty$ identically $1$ in a neighborhood of $0$),

   \item[-]
   if $R^*$ operates in the lifting variables, so that $R^*(\theta(\eta/j))$
 always gives out at least $1/j$,

   \item[-]
    and if a dominated
 convergence can apply.\medskip
 \end{itemize}
  When we shall deal with $\dela$-homogeneous
 operators, this will be made possible if the cut-off functions $\theta_j$
 are further adapted to the homogeneous structure.\medskip

 The above argument justifies the following definition
 of ``saturable Lifting''; immediately after the technicalities, we
 show (see Remark \ref{sec.one:rem_propertyIII}) that a saturable Lifting
 is always available in meaningful cases (as for $\dela$-homogeneous operators: one of our main results here).
\begin{defi}[Saturable Lifting]\label{defi.satLift}
 Let $P$ be a smooth linear PDO on $\R^n$,
 and $\widetilde{P}=P+ R$ be a lifting of $P$ on
 $\R^n\times\R^p$ as in \eqref{sec.one:rem_eqformP}.

 We say that $\widetilde{P}$ is a \emph{saturable lifting} for $P$
 if the following conditions hold:
\begin{itemize}
  \item[(S.1)]
    Every summand of the formal adjoint $R^*$ of $R$
    operates at least once in the $\xi$ variables,
    i.e., $R^*$ has the form
\begin{equation}\label{sec.one:rem_eqformPstar}
 R^*=\sum_{\beta \neq 0}
   r^*_{\alpha,\beta}(x,\xi)\,D^\alpha_x D^\beta_{\xi},
\end{equation}
 for (finitely many, possibly vanishing) smooth functions $r^*_{\alpha,\beta}(x,\xi)$.

 \item[(S.2)]
 There exists a
 sequence $(\theta_j)_j$ in $C_0^{\infty}(\R^p,[0,1])$
 such that\footnote{By this we mean that,
 denoting by $\Omega_j$
 the set $\{\xi\in \R^p:\theta_j(\xi)=1\}$, one has
 $$\bigcup_{j\in \mathbb{N}} \Omega_j=\R^p\quad \text{and}\quad
 \Omega_j\subset \Omega_{j+1}\quad
 \text{for any $j\in\mathbb{N}$}.$$}
 $\{\theta_j=1\}\uparrow \R^p$ as $j\uparrow \infty$; moreover,
 for every compact set $K\subset \R^n$
 and for any  coefficient function $r^*_{\alpha,\beta}$ of $R^*$ as in
 \eqref{sec.one:rem_eqformPstar} one can find constants $C_{\alpha,\beta}(K)$
 such that
\begin{equation} \label{sec.one:eq_mainassumptiontheta}
   \Big|r^*_{\alpha,\beta}(x,\xi)\Big(\frac{\de}{\de\xi}\Big)^\beta\theta_j(\xi)\Big|
   \leq C_{\alpha,\beta}(K),\qquad \text{for every $x\in K$, $\xi\in \R^p$ and $j \in \mathbb{N}$.}
\end{equation}
\end{itemize}
\end{defi}
  We next give some sufficient conditions for a lifting to be saturable.
 In what follows we always assume that $P$ is a smooth linear PDO on $\R^n$, and that
 $\widetilde{P} = P+ R$ is a lifting of $P$ on
 $\R^n\times\R^p$, with $R$ as in \eqref{sec.one:rem_eqformP}.
 The notation $(x,\xi)$ for the points of $\R^n\times\R^p$ is always understood.
\begin{rem}\label{sec.one:rem_propertyIII}
 (a)\,\, \emph{If the coefficients of $R$ are independent of
  $\xi$, then $\widetilde{P}$ is a saturable lifting for $P$.}
    In fact, under this assumption, the operator $R$ takes the form
    $$R = \sum_{\beta \neq 0}
   r_{\alpha,\beta}(x)\,D^\alpha_x D^\beta_{\xi},$$
    and thus its formal adjoint $R^*$ acts on smooth functions $\psi$ as follows:
    \begin{equation} \label{sec.one:eq_Pstarrem}
     \begin{split}
     R^*\psi & = \sum_{\beta \neq 0}
   (-1)^{|\alpha|+|\beta|}\,
   D^\alpha_x\Big(r_{\alpha,\beta}(x)\,D^\beta_{\xi} \psi(x,\xi)\Big)
   =:\sum_{\beta \neq 0}
   r^*_{\alpha,\beta}(x)\,D^\alpha_x D^\beta_{\xi}\psi(x,\xi).
    \end{split}
    \end{equation}
    Thus  condition (S.1) in
    Definition \ref{defi.satLift} is fulfilled.
   In order to verify (S.2) as well, we choose a function
   $\theta \in C_0^{\infty}(\R^p,[0,1])$ such that
   $\theta \equiv 1$ on the Euclidean ball
   centered at $0$ and radius $1$, and we set
    $\theta_j(\xi) := \theta\big(\xi/j\big)$, for any $\xi \in \R^p$ and any
    $j \in \mathbb{N}$. Clearly, $\{\theta_j=1\}\uparrow \R^p$ as $j\uparrow \infty$.
   Finally, for every fixed compact set $K\subseteq \R^n$
   we have
   $$\big|r^*_{\alpha,\beta}(x)\, D^\beta_\xi
   \theta_j(\xi)\big|
   \leq (1/j)^{|\beta|}\max_{K}\big|r^*_{\alpha,\beta}\big|
   \,\max_{\R^p}|D^\beta\theta|,$$
   and \eqref{sec.one:eq_mainassumptiontheta} follows.\medskip

   (b)\,\, If, for every compact set $K\subseteq \R^n$, the coefficient
   functions of the operator $R^*$ are bounded on $K\times \R^p$, then
  (S.2) of Definition \ref{defi.satLift}
   is satisfied. It suffices to take
   $\theta_j(\xi) = \theta(\xi/j)$ as in (a) above.\medskip

 (c)\,\, If $\LL$ is a smooth second-order operator
  on $\R^n$ and if we consider the associated
 Heat-type operator $\mathcal{H}=\LL-\de_t$ in $\R^n\times \R$, then we have
 the above formalism with $R=-\de_t$. Since $R$ has constant coefficients, we are
 in case (a) above and $\mathcal{H}$ is therefore a saturable Lifting
  of $\LL$.\medskip

   (d)\,\, As we shall prove in Section \ref{sec:changeofv},
   \emph{if $\Lop$ is a sum of squares of H\"ormander vector fields
    which are $\dela$-homogeneous of degree $1$ w.r.t.\,a family of dilations
    $\dela$ as in \eqref{intro.dela},
    then there exists a saturable lifting $\widetilde{\Lop}$ of $\Lop$},
    which is actually a sub-Laplacian on a suitable Carnot group $\G$
    on $\RN$. This fact is non-trivial and it will be proved in
    Theorem \ref{sec.two_1:thm_regliftLopX}.
\end{rem}
 We now prove the following useful theorem.
  \begin{thm}\label{sec.one:thm_Main}
  Let $P$ be a smooth linear PDO on $\R^n$
  and let $\widetilde{P}$ be a  saturable Lifting of
  $\Lop$ on $\R^n\times\R^p$, according to Definition
  \ref{defi.satLift}.
  Let us assume that
  there exists a fundamental solution $\widetilde{\Gamma}$ for $\widetilde{P}$
  on the whole of $\R^n\times\R^p$ (see Definition \ref{sec.one:defi_Green}),
  further satisfying the following properties:
  \begin{itemize}
   \item[(i)] for every fixed $x, y \in \R^n$ with $x \neq y$, one has
   \begin{equation} \label{sec.one:eq_sumtildeGamma_1}
   \eta \mapsto \widetilde{\Gamma}\big(x,0;y,\eta\big) \quad
   \text{belongs to}\quad L^1(\R^p);
   \end{equation}

   \item[(ii)] for every fixed $x \in \R^n$ and every compact set
   $K\subseteq\R^n$, one has
   \begin{equation} \label{sec.one:eq_sumtildeGamma_2}
   (y,\eta)\mapsto
   \widetilde{\Gamma}\big(x,0;y,\eta\big)\quad
   \text{belongs to}\quad L^1(K\times\R^p);
   \end{equation}
  \end{itemize}
   Then the function $\Gamma: \{(x;y)\in\R^n\times\R^n: x \neq y\}\longto \R$
   defined by
   \begin{equation} \label{sec.one:mainThm_defGamma}
    \Gamma(x;y) := \int_{\R^p}\widetilde{\Gamma}\big(x,0;y,\eta\big)\,\d\eta,
   \end{equation}
   is a global fundamental solution for $P$ on  $\R^n$.
  \end{thm}
  \begin{proof}
   First of all, thanks to \eqref{sec.one:eq_sumtildeGamma_1},
   $\Gamma$ is well-posed. In order to prove that
   $\Gamma$ is a fundamental solution for $P$ on $\R^n$,
   we have to prove the following facts: for every fixed $x \in \R^n$, one has
\begin{enumerate}
  \item  $\Gamma(x;\cdot) \in \Lloc{\R^n}$;
  \item $P\Gamma(x;\cdot) = -\mathrm{Dir}_x$ in $\mathcal{D}'(\R^n)$.
\end{enumerate}
  Assertion (1) is a trivial consequence of assumption  \eqref{sec.one:eq_sumtildeGamma_2}.
  We next prove assertion (2).

  To this end, we fix a point $x \in \R^n$ and a function
   $\varphi \in C_0^{\infty}(\R^n)$. Moreover, the Lifting
   $\widetilde{P}$ being saturable, there exists a sequence of test functions
   $\theta_j$ as in Definition \ref{defi.satLift}.
   Since the function $\widetilde{\Gamma}$
   is a fundamental solution for $\widetilde{P}$ on $\R^n\times \R^p$, we have
   (for large $j$'s so that $\theta_j(0)=1$)
   \begin{equation*}
   \begin{split}
    \int_{\R^n\times\R^p}\widetilde{\Gamma}\big(x,0;y,\eta\big)&\,
    \widetilde{P}^*\big(
    \varphi(y)\,\theta_j(\eta)\big)\,\d y\,\d\eta
    = -\varphi(x)\,\theta_j(0)= -\varphi(x);
    \end{split}
   \end{equation*}
   thus, recalling that $\widetilde{P} = P + R$
   (where $R$ is a linear PDO operating in $y$ and $\eta$),
   we get
   \begin{equation} \label{sec.one:eq_passtolim}
   \begin{split}
    -\varphi(x) & = \int_{\R^n\times\R^p}
    \widetilde{\Gamma}\big(x,0;y,\eta\big)\,
    \theta_j(\eta)\,P^*\varphi(y)\,\d y\d \eta \\
    & \quad\quad +
    \int_{\R^n\times\R^p}\widetilde{\Gamma}\big(x,0;y,\eta\big)\,
    R^*\big(\varphi(y)\,
    \theta_j(\eta)\big)\,\d y\,\d\eta =: \mathrm{I}_{j}+\mathrm{II}_{j},
   \end{split}
   \end{equation}
   with the obvious notation.
    Our aim is now to pass to the limit for $j\to\infty$ in
    \eqref{sec.one:eq_passtolim}. To this end we first notice that, if we denote
   by $K$ the support of the function
   $\varphi$, then both integrals expressing $\mathrm{I}_{j}$ and $\mathrm{II}_{j}$
   are actually performed over $K\times\R^p$.
   As for $\mathrm{I}_{j}$,
   a simple dominated convergence argument
   based on \eqref{sec.one:eq_sumtildeGamma_1} and \eqref{sec.one:eq_sumtildeGamma_2} shows that
\begin{equation}\label{sec.one_eq_limj1}
     \lim_{j\to\infty}\mathrm{I}_{j} =
     \int_{\R^n} \Gamma(x;y)\,P^*\varphi(y)\,\d y.
    \end{equation}
    We next turn to $\mathrm{II}_{j}$.
    First we observe that,
    since the sets $\{\eta:\theta_j(\eta)=1\}$
    increasingly invade $\R^p$,
    and since the operator $R^*$ always differentiate w.r.t.\,$\eta$
    (see (S.1) in the definition of saturable Lifting),
    we obtain that
    $$\lim_{j\to\infty}
    R^*\big(\varphi(y)\,
    \theta_j(\eta)\big) = 0, \quad \text{pointwise for $(y,\eta)\in
    K\times\R^p$}.$$
    Moreover, by writing $R^*$  as in \eqref{sec.one:rem_eqformPstar}, we get
    $$\big|R^*\big(\varphi(y)\,
    \theta_j(\eta)\big)\big|
    \leq \sum_{\beta \neq 0}
    \big|r^*_{\alpha,\beta}(y,\eta)\big|\cdot\big|D^\alpha_y\varphi(y)\big|
    \cdot\big|D^\beta_{\eta}\theta_j(\eta)\big|
    \leq C(\varphi)\,
    \sum_{\beta \neq 0}
    \big|r^*_{\alpha,\beta}(y,\eta)\,D^\beta_{\eta}\theta_j(\eta)\big|
    .$$
    From this, by crucially exploiting property \eqref{sec.one:eq_mainassumptiontheta}
    of the sequence $\theta_j$, we infer the existence
    of a positive constant $C=C(\varphi,K) > 0$ such that
\begin{equation*}
    \begin{split}
    \big|\widetilde{\Gamma}\big(x,0;y,\eta\big)\,
    R^*\big(\varphi(y)\,
    \theta_j(\eta)\big)\big|\leq
   C \,|\widetilde{\Gamma}\big(x,0;y,\eta\big)|,
    \end{split}
    \end{equation*}
    uniformly for $(y,\eta)\in K\times\R^p$ and $j \in \mathbb{N}$.
 Therefore, due to property \eqref{sec.one:eq_sumtildeGamma_2} of $\widetilde{\Gamma}$,
    we can apply once again the a dominated convergence argument to infer that
\begin{equation} \label{sec.one_eq_limj2}
     \lim_{j\to\infty}\mathrm{II}_{j} = 0.
\end{equation}
    By gathering together
    \eqref{sec.one_eq_limj1} and \eqref{sec.one_eq_limj2}, we can pass to the limit for $j\to\infty$ in
    \eqref{sec.one:eq_passtolim}, obtaining
    $$-\varphi(x) =  \int\Gamma(x;y)\,P^*\varphi(y)\,\d y.$$
    This ends the proof.
\end{proof}
 \begin{rem}\label{sec.one:rem_summtildeGamma2}
    We observe that, if $\widetilde{\Gamma}$ is a fundamental solution
    for $\widetilde{P}$ on $\R^n\times \R^p$, then we have, for every fixed $x\in\R^n$
    (see Definition \ref{sec.one:defi_Green}-(i)),
  \begin{equation} \label{sec.one:eq_remGammaLloc}
   (y,\eta)\mapsto \widetilde{\Gamma}\big(x,0;y,\eta\big) \in \Lloc{\R^n\times \R^p}.
   \end{equation}
   This means that the integrability assumption \eqref{sec.one:eq_sumtildeGamma_2}
   in Theorem \ref{sec.one:thm_Main} is actually an integrability condition
   at infinity; thus \eqref{sec.one:eq_sumtildeGamma_2} is equivalent to the following condition:\medskip

  \noindent \emph{for every $x \in \R^n$ and every
  compact set $K\subseteq \R^n$, there exists a compact set $K' \subseteq \R^p$
   such that
   $$(y,\eta)\mapsto \widetilde{\Gamma}\big(x,0;y,\eta\big)\quad
   \text{belongs to}\quad L^1\big(K\times
   (\R^p\setminus K')\big).$$}
  \end{rem}
 The next task is to consider the other properties (ii)-to-(v) possibly required to a well-behaved fundamental solution,
 and to find sufficient conditions on $\widetilde{\Gamma}$ in such a way that
 these are inherited by the $\Gamma$ function in \eqref{sec.one:mainThm_defGamma}.  Clearly, the nonnegativity
 property (ii) of $\widetilde{\Gamma}$ passes to $\Gamma$; also, if $\widetilde{\Gamma}>0$ then the same is true of $\Gamma$.
 In the following result we study condition (iv) of vanishing at infinity.
  \begin{prop} [Continuity and limit at infinity] \label{sec.one:prop_contVanishGamma}
   Let the notation and the hypotheses of Theorem
   \ref{sec.one:thm_Main} apply. Let us assume, in addition, that
   $\widetilde{\Gamma}$ satisfies the following bound property: \medskip

   \emph{(B)}\,\, For every fixed $x \in \R^n$, there
   exist a compact set $K_x\subseteq \R^p$
   and a nonnegative function $g_{x}\in L^1(\R^p\setminus K_x)$ such that
   \begin{equation} \label{sec.one:eq_AssBg}
    \widetilde{\Gamma}\big(x,0;y,\eta\big) \leq g_{x}(\eta), \quad
    \text{for every $y\in \R^n$ and every $\eta \in \R^p\setminus K_x$.}
   \end{equation}
   Then the following facts hold true:
   \begin{itemize}
    \item[(a)] if, for every fixed $x \in \R^n$, the function
    $(y,\eta)\mapsto \widetilde{\Gamma}\big(x,0;y,\eta\big)$ is continuous away from $(x,0)$,
    then the function $y\mapsto \Gamma(x;y)$ is continuous on $\R^n\setminus\{x\}$;\medskip

    \item[(b)] if, for every fixed $x \in \R^n$, the function
    $(y,\eta)\mapsto \widetilde{\Gamma}\big(x,0;y,\eta\big)$ vanishes at infinity, then the
    same is true of $y\mapsto \Gamma(x;y)$.
   \end{itemize}
  \end{prop}
  \begin{proof}
   Let us prove (a). We fix a point $y_0 \in \R^n\setminus\{x\}$
   and a real $\rho > 0$
   such that the Euclidean ball $B_\rho(y_0)$ centered at $y_0$ and radius $\rho$ is contained in $\R^n\setminus\{x\}$.
   Moreover, we choose a sequence
   $(y_j)_j$ in this ball
   converging to $y_0$ as $j\to\infty$.
   If $K_x \subseteq \R^p$ is as in assumption (B), for every $j \in \mathbb{N}$ we have
\begin{equation} \label{sec.one:eq_propTopassLim}
   \begin{split}
   \Gamma(x;y_j) &=   \int_{K_x}\widetilde{\Gamma}\big(x,0;y_j,\eta\big)\,\d \eta
   + \int_{\R^p\setminus K_x}\widetilde{\Gamma}\big(x,0;y_j,\eta\big)\,\d \eta.
   \end{split}
   \end{equation}
   We pass to the limit as $j\to\infty$ in the right-hand
   side of \eqref{sec.one:eq_propTopassLim}. To this end we first
   observe that under condition (a) we obviously have
   $\lim_{j\to\infty}\widetilde{\Gamma}(x,0;y_j,\eta)
   = \widetilde{\Gamma}(x,0;y_0,\eta)$ for every $\eta \in \R^p$.
   Moreover, since the set ${K} :=
   \overline{B_\rho(y_0)}\times K_x$ is compact,
   there exists a positive real
   constant $M_x > 0$ such that
   $$\widetilde{\Gamma}(x,0;y_j,\eta)\leq
   M_x, \quad \text{for every $j\in\mathbb{N}$ and every $\eta\in K_x$}.$$
  By a dominated convergence argument, we obtain
   \begin{equation} \label{sec.one:eq_limIntoverKx}
    \lim_{j\to\infty}\int_{K_x}\widetilde{\Gamma}\big(x,0;y_j,\eta\big)\,\d \eta
    = \int_{K_x}\widetilde{\Gamma}\big(x,0;y_0,\eta\big)\,\d \eta.
   \end{equation}
   As for the second integral in the rhs  of
   \eqref{sec.one:eq_propTopassLim}, assumption (B) is shaped in such a way that another
   dominated convergence argument can apply, so that
\begin{equation} \label{sec.one:eq_limIntoutKx}
    \lim_{j\to\infty}
    \int_{\R^p\setminus K_x}\widetilde{\Gamma}\big(x,0;y_j,\eta\big)\,\d \eta
    =
    \int_{\R^p\setminus K_x}\widetilde{\Gamma}\big(x,0;y_0,\eta\big)\,\d \eta.
\end{equation}
   By gathering together identities \eqref{sec.one:eq_limIntoverKx}
   and \eqref{sec.one:eq_limIntoutKx} we obtain the continuity of $\Gamma(x;\cdot)$ at $y_0$.

   The proof of (b) is completely analogous and it is skipped.
  \end{proof}
 We next take into account condition (v):
  \begin{prop}[Pole of $\Gamma$]\label{sec.one:prop_poleGamma}
   Let the notation and the hypotheses of Theorem
   \ref{sec.one:thm_Main} apply. Let us assume, in addition, that $\widetilde{\Gamma}$ enjoys
   the following properties:
\begin{itemize}
   \item[(a)] $\widetilde{\Gamma}$ is nonnegative;

  \item[(b)] for every $x \in \R^n$, the function
   $(y,\eta)\mapsto \widetilde{\Gamma}(x,0;y,\eta)$ is lower semi-continuous
   outside $(x,0)$;

  \item[(c)] for every $x \in \R^n$, the function
   $\eta\mapsto\widetilde{\Gamma}\big(x,0;x,\eta)$
   is not integrable on $\R^p$.
\end{itemize}
  Then the function $y\mapsto \Gamma(x;y)$
  defined in \eqref{sec.one:mainThm_defGamma} has a pole at the point $x$, i.e.,
  $$\lim_{y\to x} \Gamma(x;y)=\infty.$$
  \end{prop}
  \begin{proof}
   Let $(y_j)_j$ be a sequence in $\R^n\setminus\{x\}$
   converging to $x$ as $j\to\infty$. Since, by our assumptions, the function
   $(y,\eta)\mapsto\widetilde{\Gamma}(x,0;y,\eta)$
   is nonnegative and lower semi-continuous on $\R^n\times \R^p\setminus\{(x,0)\}$,
   from Fatou's lemma we obtain
   \begin{align*}
   \liminf_{j\to\infty}\Gamma(x;y_j)& \geq \int_{\R^p}
   \liminf_{j\to\infty} \widetilde{\Gamma}\big(x,0;y_j,\eta\big)\,\d\eta
   \geq \int_{\R^p}\widetilde{\Gamma}\big(x,0;x,\eta\big)\,\d \eta.
   \end{align*}
 Taking into account hypothesis (c), the proof is complete.
  \end{proof}
  In the sequel, we shall apply the results of the present section to PDOs of the following form:
  we consider a sum of squares $\LL=\sum_{j=1}^m X_j^2$ of H\"ormander vector fields
  which are $\dela$-homogeneous of degree $1$ w.r.t.\,a family of dilations
  $\dela$ as in \eqref{intro.dela}. We want to prove that $\LL$ can be lifted
  to a sub-Laplacian on a Carnot group by means of a saturable Lifting, in the sense of
  Definition \ref{defi.satLift}.

  To this end, we firstly need to recall Folland's version, \cite{Folland2},
  of Rothschild and Stein Lifting for homogeneous vector fields: this is done in  Section
  \ref{sec.two}. Secondly, we prove that we can perform a change of variable
  giving a saturable Lifting, which is provided in Section \ref{sec:changeofv}.
  Finally, in Section \ref{sec:GreenHOM} we shall show that all the hypotheses
  of Theorem \ref{sec.one:thm_Main} are satisfied: this will prove that
  $\LL$ admits a fundamental solution, which turns out have further selected properties (ii)-(iii)-(iv).
\section{Recalls on the Lifting of homogeneous sums of squares} \label{sec.two}
 Let us fix a family $\{X_1,\ldots,X_m\}$ of linearly
  independent smooth vector fields on Euclidean space $\R^n$, satisfying the following
  properties:
  \begin{itemize}
    \item[(H1)]
  $X_1,\ldots,X_m$ are $\delta_\lambda$-homogeneous of degree
  $1$ with respect to a family of non-isotropic dilations $\{\delta_{\lambda}\}_{\lambda > 0}$
  as in \eqref{intro.dela};

     \item[(H2)]
  $X_1,\ldots,X_m$ satisfy H\"ormander's rank
  condition at $0$, i.e.,
  $$\dim\big\{X(0) : X \in \mathrm{Lie}\{X_1,\ldots,X_m\}\big\} = n.$$
  \end{itemize}
  \begin{rem}\label{sec.two_1:remHomX}
  For a future reference, we remark that the homogeneity assumption
  (H1) is equivalent to any of the following facts:
  \begin{itemize}
   \item If $X_j = \sum_{k = 1}^na_{k,j}(x)\,\de_{x_k}$,
   the function $a_{k,j}$ is
   $\delta_\lambda$-homogeneous of degree $\sigma_k - 1$.

   \item For every fixed
   $j \in \{1,\ldots,m\}$, the following identity holds true
   \begin{equation} \label{sec.two_1:charactHomog}
    \delta_\lambda (X_j(x)) = \lambda\,X\big(\delta_\lambda(x)\big),
    \quad \text{for every $x \in \R^n$ and every $\lambda>0$}.
   \end{equation}
  \end{itemize}
 As a consequence, it is important to highlight that $a_{k,j}$ is a polynomial and
 it is independent of $x_k,\ldots,x_n$. This last fact ensures that the
 vector fields $X_1,\ldots,X_m$ are complete.
\end{rem}

   Our main goal is to prove the following theorem, by using Folland's results in \cite{Folland2}
   plus a change of variable (later introduced in Section \ref{sec:changeofv}).
\begin{thm}\label{sec.two_1:thmLiftMain}
   Let $N=\mathrm{dim}(\mathrm{Lie}\{X_1,\ldots,X_m\})$.
   There exists a homogeneous Carnot group
   $\G=\big(\RN,*,D_{\lambda}\big)$, with $m$ generators and nilpotent of step $r$,
   and there exists a system $\{Z_1,\ldots,Z_m\}$ of Lie-generators of $\mathrm{Lie}(\G)$, such
   that (for every $i = 1,\ldots,m$) $Z_i$ is a lifting of $X_i$, via the projection of $\pi:\R^N\to\R^n$
   onto the first $n$ variables.
\end{thm}
 The proof if this theorem will be constructive, modulo the knowledge
 of the Campbell-Baker-Hausdorff operation.
 To begin with,  let $\mathfrak{a}$ be the Lie algebra generated by $X_1,\ldots,X_m$:
  $$\mathfrak{a} := \mathrm{Lie}\big\{X_1,\ldots,X_m\big\}.$$
  It follows from the homogeneity assumption (H1) that every
  commutator of $X_1,\ldots,X_m$ containing more than $\sigma_n$ terms vanishes
  identically, hence $\mathfrak{a}$ is {nilpotent} of step $r \leq \sigma_n$.
  Moreover, the rank condition (H2) ensures that $r$ cannot be smaller than
  $\sigma_n$, so that $\mathfrak{a}$ is nilpotent
  of step equal to $\sigma_n$, which is therefore an integer which we also denote by $r$.

  As a consequence, $\mathfrak{a}$ being finitely generated, its dimension
  (as a subspace of the linear space of the smooth vector fields on $\R^n$) is {finite}.
  We then set
  $$N := \mathrm{dim}\big(\mathfrak{a}\big)\quad \text{and}\quad p := N - n,$$
  and we assume from now on that $ N>n$.
  Now, since $\mathfrak{a}$ is generated by
  $X_1,\ldots,X_m$ and since it is nilpotent of step $r$, we have
  \begin{equation} \label{sec.two_1:astratified}
   \mathfrak{a} = \mathfrak{a}_1\oplus\cdots\oplus\mathfrak{a}_{r},
   \quad \text{with} \quad \begin{cases}
   \mathfrak{a}_1 := \mathrm{span}\big\{X_1,\ldots,X_m\}, \\
   \mathfrak{a}_k := [\mathfrak{a}_1,\mathfrak{a}_{k-1}]\quad \text{for $2 \leq k \leq r$}; \\
   [\mathfrak{a}_1,\mathfrak{a}_{r}] = \{0\}.
   \end{cases}
  \end{equation}
  In other words, the Lie algebra $\mathfrak{a}$ is {stratified}.
  In particular,
  a vector field $X \in \mathfrak{a}$ belongs to $\mathfrak{a}_k$
  (with $1 \leq k \leq r$) {if and only}
  if $X$ is $\delta_\lambda$-homogeneous of degree $k$.

  By means of \eqref{sec.two_1:astratified},
  we can define a family $\{\Delta_\lambda\}_{\lambda > 0}$ of dilations
  on $\mathfrak{a}$ in the following way:
  \begin{equation} \label{sec.two_1:def_Deltaa}
   \Delta_\lambda(X) = \sum_{k = 1}^{r}
   \lambda^k\,V_k, \quad \text{where $X =\sum_{k=1}^r V_k$ and $V_k\in \mathfrak{a}_k$ for any $k=1,\ldots,r$}.
  \end{equation}
  Moreover, since $\mathfrak{a}$ is nilpotent, the
  {Campbell-Baker-Hausdorff multiplication} $\diamond$ in \eqref{CBHoper}
  defines a group on $\mathfrak{a}$.
  We now transfer the operation $\diamond$ and the dilation
  $\{\Delta_\lambda\}_{\lambda > 0}$ to a copy of $\mathfrak{a}$ by fixing
  a suitable coordinate system on (the finite-dimensional vector space)
  $\mathfrak{a}$.

  To this end we first observe that,
  by means of \eqref{sec.two_1:astratified}
  and of the rank condition (H2), we can complete
  $X_1,\ldots,X_m$ to form a {basis}
 $\mathcal{A} = \{X_1,\ldots,X_m,X_{m+1},\ldots,X_N\}$
  of $\mathfrak{a}$ such that:
  \begin{itemize}
   \item\label{sec.two_1:prp1basis}
 the set $\{X_1(0),\ldots,X_N(0)\}$ is a set
   of \emph{generators} for the vector space $\R^n$;

   \item   \label{sec.two_1:prp2basis}
   the basis $\mathcal{A}$ is \emph{adapted} to the
   stratification:
  $\mathcal{A} = \big\{X^{(1)}_1,\ldots,X^{(1)}_{m_1},\ldots,X^{(r)}_1,
   \ldots,X^{(r)}_{m_{r}}\big\},$
   where $m_1 = m$, $X^{(1)}_j = X_j$ for every $j = 1,\ldots, m$ and, for every
   $k=2,\ldots,r$,
   $$m_k=\dim\big(\mathfrak{a}_k\big) \quad \text{and} \quad
   \mathfrak{a}_k = \mathrm{span}\big(\{X^{(k)}_1,\ldots,X^{(k)}_{m_k}\}\big).$$
  \end{itemize}
  We then consider the linear isomorphism $\Phi $ associated
  with the basis $\mathcal{A}$, i.e.,
  $$\Phi : \RN\longto\mathfrak{a}, \quad \Phi (a):=
  \sum_{j = 1}^Na_j\,X_j.$$
  In the sequel we also set, for brevity,
  $a\cdot X := \sum_{j = 1}^Na_j\,X_j$.
  Next we define an operation $*$ and a family of dilations
  $\{D_\lambda\}_{\lambda > 0}$ on $\RN$ by pushing $\diamond$ and $D_\lambda$:
  \begin{align}
  a*b & := \Phi ^{-1}\big(\Phi (a)
  \diamond\Phi (b)\big), \quad \text{for every $a,b \in \RN$,}
  \label{sec.two_1:def_opstar}
  \\
  D_\lambda & : \RN\longto\RN, \quad D_\lambda(a) :=
  \Phi ^{-1}(\Delta_\lambda(\Phi (a))).
  \label{sec.two_1:def_Dlambda}
  \end{align}
\begin{rem}\label{sec.two_1:rem_idstarD}
 The following facts hold:
  \begin{itemize}
   \item[(a)] For every $a,b \in \RN$, the operations
   $*$ and $\diamond$ are related by the identity
   \begin{equation} \label{sec.two_1:eq_idstar}
    (a * b)\cdot X 
    =   (a \cdot X)\diamond(b \cdot X).
   \end{equation}

   \item[(b)] For every $\lambda > 0$ and every $a \in \RN$,
   the dilations $D_\lambda$ and $\Delta_\lambda$ are related
   by the identity
   \begin{equation} \label{sec.two_1:eq_idDlambda}
     D_\lambda(a)\cdot X 
      = \Delta_\lambda(a\cdot X).
   \end{equation}
\end{itemize}
   As a consequence of the above identity
   \eqref{sec.two_1:eq_idDlambda}, the dilation $D_\lambda$ can be written
   as follows
   $$D_\lambda(a) = (\lambda^{s_1}a_1,\ldots,\lambda^{s_N}a_N), \quad
   \text{for every $a \in \RN$}, $$
   where $1 = s_1 \leq \ldots \leq s_N$ are consecutive integers between
   $1$ and $r$, and
   $$\big(s_1,\ldots,s_N\big) = \big(\underbrace{1,\ldots,1}_{m},
   \underbrace{2,\ldots,2}_{m_2},\ldots,
   \underbrace{r,\ldots,r}_{m_{r}}\big).$$
   With this notation, $X_1,\ldots,X_N$
   are $\delta_\lambda$-homogeneous of degrees $s_1,\ldots,s_N$ respectively,
   and one has
   $$\Delta_\lambda(X_i) = \lambda^{s_i}\,X_i, \quad \text{for every
   $i = 1,\ldots,N$}.$$
  \end{rem}
  \noindent As it is reasonable to expect, the following fact holds true (for a proof see \cite[Theorem 17.4.2]{BLU}):\medskip

  \emph{The triple $\mathbb{A} = (\RN,*,D_\lambda)$ is a
   homogeneous Carnot group on $\RN$, with $m$ generators and nilpotent of
   step $r$.
   Furthermore, the Lie algebra $\mathrm{Lie}(\mathbb{A})$
   of $\mathbb{A}$ is isomorphic to $\mathfrak{a}$.}\medskip

 Following Folland \cite{Folland2}, we consider the crucial map
  \begin{equation} \label{sec.two_1:def_pi}
   \pi: \RN\longto \R^n, \quad \pi(a) := 
   \Psi^{a\cdot X}_t(0),
  \end{equation}
  where, for every fixed vector field $V \in \mathfrak{a}$,
  we denote by $\Psi^V_t(0)$ the integral curve at time $t$
  of the vector field $V$ starting from $0\in \R^n$ at time $0$.
  We also use the notation $\exp(tV)(0)$ for $\Psi^V_t(0)$. We explicitly
  observe that $\pi$ is well-posed, since {any vector field}
  in $\mathfrak{a}$ is complete (see Remark \ref{sec.two_1:remHomX}).
  The selected properties of $\pi$ are given in the following result, proved
  in \cite{Folland2}; we give the proof for the reading convenience.
\begin{thm}[Folland, \cite{Folland2}]\label{sec.two_1:lempropPi}
  The map $\pi$ defined in
   \eqref{sec.two_1:def_pi} satisfies the following properties.
   \begin{enumerate}
    \item For every fixed $\lambda > 0$, one has
    \begin{equation} \label{sec.two_1:eq_homPi}
     \pi\big(D_\lambda(a)\big) = \delta_\lambda\big(\pi(a)\big),
     \quad \text{for every $a \in \RN$}.
    \end{equation}

    \item $\pi$ is a surjective polynomial map.

    \item
     Let $J_1,\ldots,J_N$ be the (unique) vector fields in $\mathrm{Lie}(\mathbb{A})$
    coinciding at $0\in \RN$ with the coordinate partial derivatives; then, for every $j = i,\ldots, N$, one has
\begin{equation} \label{sec.two_1:eq_liftone}
    \d\pi (J_i)(a) = X_i(\pi(a)), \quad \text{for every $a \in \RN$}.
\end{equation}
\end{enumerate}
\end{thm}
  \begin{proof}
  (1) For every $\lambda > 0$ and every $a \in \RN$, one has
   $$\pi (D_\lambda(a) )
   \stackrel{\eqref{sec.two_1:def_pi}}{=}
   \exp (\Phi  (D_\lambda(a) ) )(0)
   \stackrel{\eqref{sec.two_1:def_Dlambda}}{=}
   \exp (\Delta_\lambda(a\cdot X) )(0),$$
   while $\delta_\lambda (\pi(a) ) = \delta_\lambda (\exp(a\cdot X )(0) )$.
   We then consider the following integral curves:
   $$\gamma(t) := \exp (t\,\Delta_\lambda(a\cdot X) )(0)
   \quad \text{and} \quad \mu(t) :=
   \delta_\lambda (\exp(t\,(a\cdot X) )(0) ), \quad
   \text{for every $t \in \R$}.$$
   One has $\gamma(0) = \mu(0) = 0$.
   Moreover, since $X_j$ is $\delta_\lambda$-homogeneous of degree $s_j$,
   \begin{align*}
    \dot \mu(t) & = \delta_\lambda ( (a\cdot X )
     (\Psi^{a\cdot X}_t(0) ) ) = \sum_{j = 1}^Na_j\delta_\lambda
     (X_j (\Psi^{a\cdot X}_t(0) ) ) \stackrel{\eqref{sec.two_1:charactHomog}}{=}
    \sum_{j = 1}^N\lambda^{s_j}a_j
    X_j (\delta_\lambda (\Psi^{a\cdot X}_t(0) ) )\\
    &= \sum_{j = 1}^N\lambda^{s_j}a_j
    X_j (\mu(t) ) =  (D_\lambda(a)\cdot X )(\mu(t))
    \stackrel{\eqref{sec.two_1:eq_idDlambda}}{=}
    \Delta_\lambda ((a\cdot X) )(\mu(t)).
   \end{align*}
   From the very definition of $\gamma$ we get
   $\dot\gamma(t) = \Delta_\lambda ((a\cdot X) )(\gamma(t)),$
   and this shows that $\gamma$ and $\mu$ solve  the same Cauchy problem, whence they coincide;
   by taking $t = 1$ we get \eqref{sec.two_1:eq_homPi}.\medskip

  (2) Clearly $\pi \in C^\infty(\RN,\R^n)$. Moreover,
   by applying Taylor's formula,
 we get
   $$\pi(a) = (a\cdot X)(0) + \mathcal{O}(\|a\|^2), \quad \text{as $a \to 0$}.$$
   This shows that the Jacobian matrix of $\pi$ at $a = 0$ is given by the matrix
   \begin{equation} \label{sec.two_1:eq_JacPi}
   \mathcal{J}_{\pi}(0) =  (X_1(0)\cdots X_N(0) ),
   \end{equation}
   and thus $\mathrm{rank} (\mathcal{J}_{\pi}(0) ) = n$.
   As a consequence,
   it is possible to find an open neighborhood $W$ of $0 = \pi(0) \in \R^n$
   such that $\pi:\RN\to W$ is surjective.
   We claim that the homogeneity property \eqref{sec.two_1:eq_homPi} implies that $\pi$
   is also onto $\R^n$. Indeed, let $x \in \R^n$ be fixed and let
   $\lambda > 0$ be such that $y = \delta_\lambda(x) \in W$.
   Since $\pi$ is onto $W$, there
   exists a point $a \in \RN$ such that $\pi(a) = y$, and thus
   $$\pi (D_{1/\lambda}(a) )
   \stackrel{\eqref{sec.two_1:eq_homPi}}{=}
   \delta_{1/\lambda} (\pi(a) )
   = \delta_{1/\lambda} (\delta_\lambda(x) ) = x,$$
   proving that $\pi$ is surjective.\medskip

  (3)  Let $i \in \{1,\ldots,N\}$ be fixed and let
  $e_i$ denote the $i$-th vector of the canonical basis of $\RN$. By definition of $J_i$,
   for every $a \in \RN$ we have
 \begin{equation*} \label{sec.two_1:eq_toinsertCH}
   \begin{split}
    \d_a\pi (J_i(a) ) &= J_i(\pi)(a) = \frac{\d}{\d t} \Big|_{t = 0}
   \pi (a*(t\,e_i) ) \stackrel{\eqref{sec.two_1:def_pi}}{=} \frac{\d}{\d t} \Big|_{t = 0} (
   \exp ((a*(t\,e_i))\cdot X )(0) ) \\
   & \!\!\stackrel{\eqref{sec.two_1:eq_idstar}}{=}
   \frac{\d}{\d t} \Big|_{t = 0} (
   \exp ((a \cdot X)\diamond ((t\,e_i)\cdot X) )(0) ) = \frac{\d}{\d t} \Big|_{t = 0} (
   \exp ((a \cdot X)\diamond (t\,X_i) )(0) ).
   \end{split}
   \end{equation*}
   We now recall that the Campbell-Baker-Hausdorff multiplication
   satisfies the formula
   \begin{equation*} \label{sec.two_1:eq_CHThmLiftone}
    \exp(W) (\exp(V)(x) ) = \exp(V\diamond W)(x), \,\,\,\text{for all
    $x \in \RN$ and every $V,W\in\mathfrak{a}$}.
   \end{equation*}
   Therefore, by inserting this in the above computation, we obtain
   \begin{align*}
    \d_a\pi (J_i(a) ) & =
    \frac{\d}{\d t} \Big|_{t = 0} (\exp(t\,X_i) (\exp(a\cdot X)(0) ) ) = X_i (\exp(a\cdot X)(0) )
    \stackrel{\eqref{sec.two_1:def_pi}}{=} X_i (\pi(a) ).
   \end{align*}
   This is precisely the desired \eqref{sec.two_1:eq_liftone},
   and the proof is complete.
  \end{proof}
\section{A change of variable turning $\pi$ into a projection, and allowing saturation}\label{sec:changeofv}
  In order to construct a projection acting as a lifting
  for $X_1,\ldots,X_m$, we add a new feature to Folland's ideas: we
  find an appropriate change of coordinates of the group $\mathbb{A}$
  in the previous section which transforms the vector fields $J_1,\ldots,J_m$ on $\mathbb{A}$ into new vector fields
  $Z_1,\ldots,Z_m$ on $\RN$ lifting $X_1,\ldots,X_m$ via the projection of $\RN$ onto $\R^n$.

  To this end we first observe that, since the vectors
  $X_1(0),\ldots,X_N(0)$ generate the whole of $\R^n$,
  we can find $n$ indexes in $\{1,\ldots,r\}$
  $$1 = i_1 < i_2 < \cdots < i_n,$$
  such that $\mathcal{B} := \{X_{i_1}(0),\ldots,X_{i_n}(0)\}$ is
  a  {basis} of $\R^n$. As a consequence,
  the vector fields
  $X_{i_1},\ldots,X_{i_n}$ must be $\delta_\lambda$-homogeneous of degree
  $\sigma_1,\ldots,\sigma_n$, respectively.
  We then set
\begin{equation}\label{sceltaindexes}
 \{j_1,\ldots,j_p\} := \{1,\ldots,r\}\setminus\{i_1,\ldots,i_n\}
  \qquad  (p = N-n ),
\end{equation}
  and we note that, from H\"ormander's rank condition
  (H2), it follows that
  $j_p \leq r - 1$,
  that is,
   {all the vector fields} in the basis $\mathcal{A}$
  which are
  $\delta_\lambda$-homogeneous of maximum degree $r=\sigma_n$
  contribute to $\mathcal{B}$.
  So far we have assumed that H\"ormander's rank condition holds at $0$ only;
  the last remark shows that it automatically holds at any point of $\R^n$.
\begin{rem}\label{sec.two_1:remHormander}
 With the above notation, we claim that
   \begin{equation} \label{sec.one_1:eq_HormRn}
    \dim \{X_{i_1}(x),\ldots,X_{i_n}(x) \} = n,
    \quad \text{for every $x \in \R^n$}.
   \end{equation}
   In order to see this, let us consider the matrix-valued function $\mathbf{M}$
   defined as follows\footnote{Here $\mathrm{M}_n(\R)$ denotes the set of the real-valued $n\times n$ matrices;
   in the definition of $\mathbf{M}(x)$, $X_{i_1}(x),\ldots, X_{i_n}(x)$ are meant as $n\times 1$ column vectors.}
   $$\mathbf{M}:\R^n \longto \mathrm{M}_n(\R), \quad
   \mathbf{M}(x) := \big(X_{i_1}(x)\cdots X_{i_n}(x)\big).$$
   Since $\{X_{i_1}(0),\ldots,X_{i_n}(0)\}$ is a basis of $\R^n$, the matrix
   $\mathbf{M}(0)$ is non-singular; therefore, it is possible to find
   a small open neighborhood $\mathcal{U}$ of $0$ (in $\R^n$) such that
   $\det(\mathbf{M}(x)) \neq 0$ for every $x \in \mathcal{U}$.
   We now fix a point $x \in \R^n$ and we choose $\lambda > 0$
   such that $\delta_\lambda(x) \in \mathcal{U}$. Then,
   recalling that $X_{i_1},\ldots,X_{i_n}$ are $\delta_\lambda$-homogeneous
   of degrees $\sigma_1,\ldots,\sigma_n$ respectively, we have
   \begin{align*}
    \mathbf{M}\big(\delta_\lambda(x)\big)
    & \,\,\,= \det\Big(X_{i_1}\big(\delta_\lambda(x)\big)\cdots
    X_{i_n}\big(\delta_\lambda(x)\big)\Big) \\
   & \stackrel{\eqref{sec.two_1:charactHomog}}{=}
    \det\Big(\lambda^{-\sigma_1}\,\delta_\lambda\big(X_{i_1}(x)\big)\cdots
    \lambda^{-\sigma_n}\,\delta_\lambda\big(X_{i_n}(x)\big)\Big) \\
    & \,\,\,= \lambda^{-\sigma_1}\cdots\lambda^{-\sigma_n}\,
    \det\Big(\delta_\lambda\big(X_{i_1}(x)\big)\cdots
    \delta_\lambda\big(X_{i_n}(x)\big)\Big),
   \end{align*}
   and thus, since the point $\delta_\lambda(x)$ belongs to $\mathcal{U}$,
   we obtain
   $$\det\Big(\delta_\lambda\big(X_{i_1}(x)\big)\cdots
    \delta_\lambda\big(X_{i_n}(x)\big)\Big) \neq 0.$$
    This ensures that the vectors
    $\delta_\lambda\big(X_{i_1}(x)\big),\ldots,
    \delta_\lambda\big(X_{i_n}(x)\big)$ form a basis of $\R^n$, so that
    the same is true of $X_{i_1}(x),\ldots,X_{i_n}(x)$, since the map
    $\delta_\lambda$ is an isomorphism of $\R^n$.

    As a consequence, we see that $X_1,\ldots,X_m$ satisfy
     H\"ormander's rank condition not only at the origin $0$
    (see assumption (H2)), but at  {every point
    of $\R^n$}, that is,
    $$\dim\big\{X(x) : X \in \mathrm{Lie}\{X_1,\ldots,X_m\}\big\} = n,
    \quad \text{for every $x \in \R^n$}.$$
  \end{rem}
 We are ready to introduce our change of coordinates: we set, with reference to \eqref{sceltaindexes},
  \begin{equation} \label{sec.two_1:def_mapT}
   T: \RN\longto\RN, \quad T(a) := \big(\pi(a),a_{j_1},\ldots,a_{j_p}\big).
  \end{equation}
  We also define a new family $\{d_\lambda\}_{\lambda > 0}$
  of dilations on $\RN$ by setting
  \begin{equation} \label{sec.one_1:def_dilnonordinate}
   d_\lambda: \RN \longto \RN, \quad d_\lambda(a) :=
   (\lambda^{\sigma_1}a_1,\ldots,\lambda^{\sigma_n}a_n, \lambda^{s_{j_1}}a_{n+1},
   \ldots,\lambda^{s_{j_p}}a_N).
  \end{equation}
  We then have the following crucial result.
  \begin{lem}\label{sec.two_1:lem_prpT}
   The map $T$ defined in \eqref{sec.two_1:def_mapT} satisfies
   the following properties:
   \begin{itemize}
    \item[(i)] For every fixed $\lambda > 0$, one has
    \begin{equation} \label{sec.two_1:eq_hompropT}
     T\big(D_\lambda(a)\big) = d_\lambda\big(T(a)\big),
     \quad \text{for every $a \in \RN$};
    \end{equation}
    \item[(ii)] The map $T$ is a $C^\infty$-diffeomorphism
    of $\RN$ onto itself.
   \end{itemize}
  \end{lem}
  \begin{proof}
   (i): For every $\lambda > 0$ and every  $a \in \RN$ we have
  \begin{equation*}
  \begin{split}
   T\big(D_\lambda(a)\big) & \stackrel{\eqref{sec.two_1:def_mapT}}{=}
   \Big(\pi\big(D_\lambda(a)\big), \big(D_\lambda(a)\big)_{j_1},\ldots,
   \big(D_\lambda(a)\big)_{j_p}\Big) \\
   & \,\,\,= \Big(\delta_\lambda\big(\pi(a)\big),
   \lambda^{s_{j_1}}a_{j_1}
   \ldots,\lambda^{s_{j_p}}a_{j_p}\Big)
   \stackrel{\eqref{sec.one_1:def_dilnonordinate}}{=}
   d_\lambda\big(\pi(a),a_{j_1},\ldots,a_{j_p}\big)= d_\lambda\big(T(a)\big),
   \end{split}
  \end{equation*}
  which is precisely the desired identity \eqref{sec.two_1:eq_hompropT}.

  (ii): Obviously  $T \in C^\infty(\RN,\RN)$. Moreover,
  $$\mathcal{J}_T(0) = \begin{pmatrix}
  \mathcal{J}_\pi(0) \\
  e_{j_1} \\
  \vdots \\
  e_{j_p}
  \end{pmatrix} \stackrel{\eqref{sec.two_1:eq_JacPi}}{=}
   \begin{pmatrix}
    X_1(0) \cdots X_N(0)
    \\
  e_{j_1} \\
  \vdots \\
  e_{j_p}
   \end{pmatrix},$$
   where $e_{j_1},\ldots,e_{j_p}$ denote some of the vectors (written as row $1\times N$ vectors) of the canonical basis of $\RN$.
   From this, by recalling that $X_{i_1}(0),\ldots,X_{i_n}(0)$ form a basis
   of $\R^n$ and by \eqref{sceltaindexes}, we derive
   that $\mathcal{J}_T(0)$ is  {invertible}, so that
   there exist neighborhoods $\mathcal{U},\mathcal{W}$ of $0$ in $\RN$ such that
   $$T|_{\mathcal{U}}: \mathcal{U} \longrightarrow \mathcal{W},
   \quad \text{is a $C^{\infty}$-diffeomorphism}.$$
   We now claim that the homogeneity property
   (i) implies that the map
   $T$ is actually a $C^\infty$-dif\-feo\-mor\-phism
   of $\RN$ onto itself.
   To prove this claim, we first show that $T$ is a bijection. \medskip

   \emph{$T$ is 1-1:}\,\,Suppose that $a,b \in \RN$ are
   such that $T(a) = T(b)$, and let $\lambda > 0$ be so small that
   $D_\lambda(a),D_\lambda(b) \in \mathcal{U}.$
   This gives
   $$T (D_\lambda(a) )
   \stackrel{\eqref{sec.two_1:eq_hompropT}}{=} d_\lambda (T(a) ) =
   d_\lambda (T(b) )
   \stackrel{\eqref{sec.two_1:eq_hompropT}}{=}
   T (D_\lambda(b) ),$$
   and thus, since $D_\lambda(a),D_\lambda(b) \in \mathcal{U}$ and $T|_{\mathcal{U}}$
   is injective, we get $D_\lambda(a) = D_\lambda(b)$, hence
   $a = b$.\medskip

   \emph{$T$ is onto:}\,\, Let $u \in \RN$ be
   fixed and let $\lambda > 0$ be such that $v = d_\lambda(u) \in \mathcal{W}$.
   Since $T|_{\mathcal{U}}$ is onto $\mathcal{W}$, it is
   possible to find a (unique)
   point $a \in \mathcal{U}$ such that $T(a) = d_\lambda(u)$, and thus
   $$T (D_{1/\lambda}(a) ) \stackrel{\eqref{sec.two_1:eq_hompropT}}{=}
   d_{1/\lambda} (d_\lambda(u) ) = u.$$
   This proves that $T$ is surjective. \medskip

   In order to end the proof, we are left show that
   the map $T^{-1}$ (which is globally defined) is smooth.
   To this end we first notice that, from the homogeneity property
   (i) of $T$, we get
   \begin{equation} \label{sec.two_1:eq_hompropinvT}
    T^{-1} (d_\lambda(u) ) = D_\lambda (T^{-1}(u) ),
    \quad \text{for every $u \in \RN$}.
   \end{equation}
   Let now $u_0 \in \RN$ and $\lambda > 0$
   be such that $d_\lambda(u_0)\in\mathcal{W}$. The map $d_\lambda$ being continuous,
   it is possible to find a positive $\rho > 0$ such that
   $d_\lambda (B(u_0,\rho) )\subseteq\mathcal{W};$
   thus, for every $u \in B(u_0,\rho)$,
   \begin{align*}
    T^{-1}(u) & = T^{-1} (d_{1/\lambda} (d_\lambda(u) ) )
     \stackrel{\eqref{sec.two_1:eq_hompropinvT}}{=}
     D_{1/\lambda} (T^{-1} (d_\lambda(u) ) ) \\
     &=  (D_{1/\lambda}\circ  (T^{-1} ) |_{\mathcal{W}}\circ d_\lambda )(u).
    \end{align*}
    This shows that $T^{-1}$  {coincides} with the smooth function
    $D_{1/\lambda}\circ  (T^{-1} )\big|_{\mathcal{W}}\circ d_\lambda$
    on the open ball $B(u_0,\rho)$, hence $T^{-1}$ is smooth near $u_0$.
    The arbitrariness of $u_0$ completes the proof.
  \end{proof}
  Thanks to Lem.\,\ref{sec.two_1:lem_prpT} we are entitled to use the change of variable
  $T$ in order to define a new homogeneous Carnot group $\G = (\RN,\star,D^\star_\lambda)$
  starting from $\mathbb{A} = (\RN,*,D_\lambda)$.

  We henceforth denote the points of $\RN = \R^n\times\R^p$
  by $(x,\xi)$, with $x\in\R^n$ and $\xi \in \R^p$; we set
  \begin{align}
   (x,\xi)&\star(y,\eta) :=
   T (T^{-1}(x,\xi)*T^{-1}(y,\eta) ),
   \quad \text{for every $(x,\xi),(y,\eta)\in\RN$}; \label{sec.two_1:eq_defoperstarG}\\
   & D^\star_\lambda: \RN\longto\RN, \quad
   D^\star_\lambda(x,\xi) := T (D_\lambda (T^{-1}(x,\xi) ) ).
   \label{sec.two_1:eq_defDstarG}
  \end{align}
   It is obvious that $\mathbb{G} = (\RN,\star,D^\star_\lambda)$ is a homogeneous
   Carnot group on $\RN$, with $m$ generators and nilpotent of
   step $r$. Furthermore, $T$ is an isomorphism between $\mathbb{A}$ and $\mathbb{G}$, that is,
   $$T(a)\star T(b) = T(a*b), \quad \text{for every $a,b \in \RN$}.$$
  We also have, for every $\lambda > 0$,
\begin{equation*}
  D^\star_\lambda (x,\xi) =
  T (D_\lambda (T^{-1}(x,\xi) ) )
  \stackrel{\eqref{sec.two_1:eq_hompropinvT}}{=}
  T (T^{-1}(d_\lambda(x,\xi) ) \\
   = d_\lambda(x,\xi).
  \end{equation*}
 There is therefore no reason to use the notation $D^\star_\lambda$
 any longer, and we replace it by $d_\lambda$. In the new coordinates $(x,\xi)$
 it is useful to write $d_\lambda(x,\xi)$ as the product $(\dela(x),\delta^*_\lambda(\xi))$, where
\begin{equation}\label{delaprodottosa}
  \delta^*_\lambda(\xi)=(\lambda^{\sigma^*_1}\xi_1,\ldots,\lambda^{\sigma^*_p}\xi_p),\quad\text{where
 $\sigma^*_i:=s_{j_i}$ for any $i=1,\ldots,p$}.
\end{equation}
  Now, since $T$ is an isomorphism of Lie groups, it induces
  the Lie algebra isomorphism $ \d T $
\begin{equation} \label{sec.two_1:eq_PsiTisom}
   \d T : \mathrm{Lie}(\mathbb{A}) \longto \mathrm{Lie}(\G),
  \quad  \d T (X)_{(x,\xi)} := \d T (X)_{T^{-1}(x,\xi)}.
\end{equation}
  We can then consider, in particular, the vector fields
  \begin{equation} \label{sec.two_1:eq_defZi}
   Z_i :=  \d T (J_i), \quad \text{for every $i = 1,\ldots,N$}.
  \end{equation}
  The map $ \d T $ being an  {isomorphism} of Lie algebras, we immediately
  infer that
  \begin{itemize}
    \item the set $\{Z_1,\ldots,Z_N\}$ is a basis of $\mathrm{Lie}(\G)$;

    \item $\mathrm{Lie}(\G) = \mathrm{Lie}\{Z_1,\ldots,Z_m\}.$
  \end{itemize}
 We can finally prove the following result.
  \begin{thm} \label{sec.two_1:thm_glifta}
   Let $Z_1,\ldots,Z_N$ be as in \eqref{sec.two_1:eq_defZi}. Then
   \begin{enumerate}
    \item[(i)] $Z_1,\ldots,Z_N$ are $d_\lambda$-homogeneous
    of degree $s_1,\ldots,s_N$ respectively
    (see Remark \ref{sec.two_1:rem_idstarD});
    \item[(ii)] $Z_i$ is a lifting
    of $X_i$, that is,
\begin{equation} \label{sec.two_1:eq_ZliftX}
     Z_i = X_i + R_i \quad (i = 1,\ldots,N),
\end{equation}
    where $R_i$ is a vector field on $\RN$ only operating in the $\xi$ variables (with coefficients possibly depending on $(x,\xi)$).
 As a consequence, the sub-Laplacian
     $\Lop_\G := \sum_{k = 1}^mZ_k^2$ on $\G$
     is a lifting of the operator $\LL= \sum_{k = 1}^m X_k^2$.
   \end{enumerate}
  \end{thm}
 Theorem \ref{sec.two_1:thm_glifta} proves Theorem
  \ref{sec.two_1:thmLiftMain}.
  \begin{proof}
  (i) We fix $i \in \{1,\ldots,N\}$ and $\lambda > 0$.
  We recall that
  \begin{equation}\label{delaJJJJ}
     \text{$J_i$ is $D_\lambda$-homogeneous of degree $s_i$.}
  \end{equation}
   For every $(x,\xi) \in \RN = \R^n\times\R^p$, we have the following computation
   \begin{align*}
    Z_i (d_\lambda(x,\xi) ) &
    \stackrel{\eqref{sec.two_1:eq_defZi}}{=}
     \d T (J_i) (d_\lambda(x,\xi) )
    \stackrel{\eqref{sec.two_1:eq_PsiTisom}}{=}
    J_i(T) (T^{-1} (d_\lambda(x,\xi) ) )
    \\
    & \stackrel{\eqref{sec.two_1:eq_defDstarG}}{=}
    J_i(T) (D_\lambda (T^{-1}(x,\xi) ) )
    \stackrel{\eqref{delaJJJJ}}{=}
   \lambda^{-s_i}\,J_i (T\circ D_\lambda ) (T^{-1}(x,\xi) )\\
   &\,\stackrel{\eqref{sec.two_1:eq_hompropT}}{=}
   \lambda^{-s_i}\,J_i (d_\lambda\circ T ) (T^{-1}(x,\xi) ) \\
   &
   \stackrel{\eqref{sec.two_1:eq_PsiTisom}}{=}
   \lambda^{-s_i}\, \d T (J_i) (d_\lambda )(x,\xi)
   =  \lambda^{-s_i}\,d_\lambda ( \d T (J_i)(x,\xi) )
   = d_\lambda (Z_i(x,\xi) ),
   \end{align*}
   and this proves that $Z_i$ is $d_\lambda$-homogeneous
   of degree $s_i$, as claimed.\medskip

  (ii) We fix $i \in \{1,\ldots,N\}$ and $(x,\xi) \in \RN$; we have
   \begin{equation} \label{sec.one_1:eq_Zitosub1}
    \begin{split}
    Z_i(x,\xi) & \,\,\,=  \d T (J_i)(x,\xi) = J_i(T) (T^{-1}(x,\xi) )
    \\
    &
    \stackrel{\eqref{sec.two_1:def_mapT}}{=}
     \Big(J_i(\pi),J_i(a\mapsto a_{j_1}),\ldots,J_i(a\mapsto a_{j_p}) \Big)
     (T^{-1}(x,\xi) ).
    \end{split}
   \end{equation}
   On the other hand, by \eqref{sec.two_1:eq_liftone} we infer
   \begin{equation} \label{sec.one_1:eq_Zitosub2}
    J_i(\pi) (T^{-1}(x,\xi))=
    X_i (\pi (T^{-1}(x,\xi) ) ).
   \end{equation}
   Now, since $T(x,\xi)=(\pi(x),\xi)$, we derive that
   \begin{equation} \label{sec.one_1:eq_Zitosub3}
   \pi(T^{-1}(x,\xi) )= x;
   \end{equation}
   therefore, by
   inserting \eqref{sec.one_1:eq_Zitosub2}
   and \eqref{sec.one_1:eq_Zitosub3} in
   \eqref{sec.one_1:eq_Zitosub1}, we infer
   $$Z_i(x,\xi)
   =  \Big(X_i(x,\xi), f_{i,1}(x,\xi),
   \ldots,f_{i,p}(x,\xi) \Big),$$
   where, for $k = 1,\ldots,p$ we have used the notation
   $$f_{i,k}(x,\xi) =
   J_i(a\mapsto a_{j_k}) (T^{-1}(x,\xi) )
   =  (J_i (T^{-1}(x,\xi) ) )_{j_k}.$$
   This shows that the vector field $Z_i$ can be written
   (as a vector field on $\RN$) in the form
   $$Z_i = X_i + R_i, \quad \text{with \,\,\,
   $R_i = \sum_{k = 1}^p f_{i,k}(x,\xi)\,\frac{\de}{\de {\xi_k}}$},$$
   hence $Z_i$ is a lifting for $X_i$ and this ends the proof.
  \end{proof}
 All the algebraic machinery of this section is motivated by the following central result.
  \begin{thm}\label{sec.two_1:thm_regliftLopX}
   The sub-Laplacian $\Lop_\G = \sum_{k = 1}^mZ_k^2$ on the homogeneous Carnot group
   $\G = (\R^n\times\R^p,\star,d_\lambda)$
   is  {a saturable Lifting} of $\LL=\sum_{k = 1}^m X_k^2$, in the sense of
   Definition \ref{defi.satLift}.
  \end{thm}
  \begin{proof}
   With reference to Definition \ref{defi.satLift}, we need
   to prove properties (S.1) and (S.2).\medskip

   (S.1) Since $Z_1,\ldots,Z_m$ are
   $d_\lambda$-homo\-ge\-neous of degree $1$, the operator
   $\Lop_\G$ is (formally) {self-adjoint} on $L^2(\RN)$.
   The same is true of $\LL $, this time invoking the
    $\dela$-homo\-ge\-neity of degree $1$ of $X_1,\ldots,X_m$.
   Thus the formal adjoint $R^*$ of $R= \Lop_\G-\LL $  {coincides} with $R$
   so that ($\Lop_\G$ being a lifting for $\LL $)
   it has the form \eqref{sec.one:rem_eqformPstar}.\medskip

   (S.2) With reference to the
   dilations $\delta^*_\lambda$  in \eqref{delaprodottosa},
   we consider $\delta^*_\lambda$-homogeneous map
\begin{equation} \label{sec.two_1:eq_defnormN}
   N:\R^p\longto\R, \quad N(\xi) := \sum_{k = 1}^p|\xi_k|^{1/\sigma^*_k}.
\end{equation}
   We now choose a smooth function $\theta \in C_0^\infty(\R^p,[0,1])$
   such that
   \begin{itemize}
     \item $\mathrm{supp}(\theta) \subseteq \{\xi \in \R^p: N(\xi) \leq 2\}$;
     \item $\theta \equiv 1$ on $\{\xi\in\R^p: N(\xi) < 1\}$.
   \end{itemize}
   We then define a sequence $\theta_j$ in $C^\infty_0(\R^p)$ by setting, for any $\xi\in \R^p$ and any $j\in \mathbb{N}$,
   \begin{equation} \label{sec.one_1:eq_defthetaj}
  \theta_j(\xi) := \theta (\delta^*_{2^{-j}}(\xi) ).
   \end{equation}
   Obviously, any $\theta_j$ is valued in $[0,1]$; furthermore, since
   $N$ is $\delta^*_\lambda $-homogeneous of degree $1$,
 \begin{itemize}
   \item $\mathrm{supp}(\theta_j) \subseteq \{\xi \in \R^p: N(\xi) \leq 2^{j+1}\}$;
   \item $\theta \equiv 1$ on $\{\xi\in\R^p: N(\xi) < 2^{j}\}$.
 \end{itemize}
 Consequently  $\{\theta_j=1\}\uparrow \R^p$ as $j\uparrow \infty$.
 In order to complete the verification of (S.2),
 let us fix a compact set $K\subseteq \R^n$ and
 let $r_{\alpha,\beta}$
  be the coordinate coefficient function of the
   PDO
   $$R^* = R = \Lop_\G-\LL  = \sum_{k = 1}^m(Z_k^2-X_k^2)
   = \sum_{\alpha, \beta}r_{\alpha,\beta}(x,\xi)\,D^\alpha_xD^\beta_\xi.$$
   The functions  $r_{\alpha,\beta}$ are polynomials;
   a simple but tedious computation shows that any monomial decomposing $r_{\alpha,\beta}(x,\xi)$,
   has the following feature: as a function of $\xi$ only it is
   $\delta^*_\lambda $-homogeneous of degree  not exceeding $|\beta|_{*}-1$, where we have used
   the notation (see also \eqref{delaprodottosa})
   $$|\beta|_{*} := \sum_{k = 1}^p\beta_k\,\sigma^*_k,
   \quad \text{for every multi-index $\beta\in(\mathbb{N}\cup\{0\})^p$}.$$
   With this notation, note that, for any $\xi\in\R^p$ and any multi-index $\beta$,
\begin{equation}\label{valoreassdelastar}
   \big(\delta^*_\lambda(\xi)\big)^\beta=\lambda^{|\beta|_*}\,\xi^\beta.
\end{equation}
    We can write $r_{\alpha,\beta}$ in the following way
   \begin{equation} \label{sec.one_1:eq_mainrepra}
    r_{\alpha,\beta}(x,\xi) = \sum_{|\gamma|_{*}\leq|\beta|_{*}-1}
    c_{\alpha,\beta,\gamma}(x)\,\xi^\gamma,
   \end{equation}
   where $c_{\alpha,\beta,\gamma}(x)$ are polynomial functions only depending on $x$.

   Now, for every multi-index $\gamma$ with
   $|\gamma|_{*} \leq |\beta|_{*} - 1$, every
   $(x,\xi)\in K\times \R^p$ and every $j \in \mathbb{N}$,
   we have the estimate (we use the notation $\chi_B$
   for the characteristic function of a set $B$):
   \begin{align}
    &\notag
    \Big|c_{\alpha,\beta,\gamma}(x)\,\xi^\gamma \,D^\beta_\xi\theta_j(\xi) \Big|
    \leq  \max_{x\in K} |c_{\alpha,\beta,\gamma}(x) | \cdot |\xi^\gamma |
    \cdot |D^\beta_\xi\theta_j(\xi) | \\
    &\notag
    (\text{recall that $\theta_j$ is  {constant} outside the set
    $B_j := \{2^{j} \leq
    N(\xi) \leq 2^{j+1}\}$} ) \\
    &\notag
     \,\,\,=  \max_{K} |c_{\alpha,\beta,\gamma} | \cdot |\xi^\gamma |
    \cdot |D^\beta_\xi\theta_j(\xi) |\cdot
    \chi_{B_j}(\xi) \\
    &\notag
     \stackrel{\eqref{sec.one_1:eq_defthetaj}}{\leq}
     \max_{K} |c_{\alpha,\beta,\gamma} |\cdot
    \sup_{\R^p} |D^\beta\theta | \cdot (2^{-j} )^{|\beta|_{*}}
    \cdot |\xi^\gamma |\cdot\chi_{B_j}(\xi) \\
    &\notag
    (\text{we denote by $\mathbf{c}_{\alpha,\beta,\gamma}$ a constant bounding the product of the first two factors,}\\
    &\notag
    \text{we write $\xi=\delta^*_{2^j}\circ \delta^*_{2^{-j}}(\xi)$ and we use \eqref{valoreassdelastar})}\\
        &\label{bellastima}
         \,\,\,\leq \mathbf{c}_{\alpha,\beta,\gamma}\cdot (2^{-j} )^{|\beta|_{*}
    -|\gamma|_{*}}
    \cdot \Big|(\delta^*_{2^{-j}}(\xi))^\gamma \Big|\cdot \chi_{B_j}(\xi).
 \end{align}
 Observe that, if the point $\xi$ belongs to the annulus
 $B_j =  \{2^{j} \leq
    N(\xi) \leq 2^{j+1}\}$, then
    the point $\delta^*_{2^{-j}}(\xi)$ belongs to the compact set
    $B_1 = \{\xi \in \R^p: 1 \leq N(\xi) \leq 2\};$
    as a consequence, there exists a constant $M_\gamma > 0$,
     {only depending on $\gamma$ but independent on $j \in \mathbb{N}$}, such that
    \begin{equation} \label{sec.two_1:eq_mainestimbis}
      |\delta^*_{2^{-j}}(\xi)^\gamma |\,
    \chi_{B_j}(\xi) \leq M_\gamma, \quad \text{for every $\xi \in \R^p$}.
    \end{equation}
    Since $|\beta|_{*}-|\gamma|_{*}\geq 1$, from \eqref{bellastima} and \eqref{sec.two_1:eq_mainestimbis},
    we then obtain
    \begin{equation} \label{sec.two_1:eq_mainestimfinal}
      |c_{\alpha,\beta,\gamma}(x)\,\xi^\gamma\,D^\beta_\xi\theta_j(\xi) |
     \leq \mathbf{c}_{\alpha,\beta,\gamma}\,M_\gamma,
     \quad \text{for every $x\in K$, $\xi\in \R^p$ and $j \in \mathbb{N}$}.
    \end{equation}
    We are now ready to conclude: by taking into account \eqref{sec.one_1:eq_mainrepra},
    for every for every $x\in K$, $\xi\in \R^p$ and $j \in \mathbb{N}$ we have
 \begin{align*}
      |r_{\alpha,\beta}(x,\xi)&\cdot D^\beta_\xi
   \theta_j(\xi) |
   \leq \sum_{|\gamma|_{*}\leq|\beta|_{*}-1}
    |c_{\alpha,\beta,\gamma}(x)\,\xi^\gamma\,D^\beta_\xi\theta_j(\xi) |
    \stackrel{\eqref{sec.two_1:eq_mainestimfinal}}{\leq}
   \sum_{|\gamma|_{*}\leq|\beta|_{*}-1}\mathbf{c}_{\alpha,\beta,\gamma}\,M_\gamma,
    \end{align*}
    and this completes the verification of property (S.2) of a saturable Lifting.
  \end{proof}
 \section{Fundamental solution  for homogeneous second-order sums of squares}\label{sec:GreenHOM}
 Throughout, $\LL  = \sum_{j = 1}^mX_j^2$ is a sum of squares of
 (linearly independent) vector fields satisfying assumptions (H1) and (H2) in
 Section \ref{sec.two}.
 Without further comments, we denote by $\G = (\RN,\star,d_\lambda)$
 the {homogeneous Carnot group} on $\RN = \R^n\times \R^p$ constructed in the previous section,
 with the sub-Laplacian $\Lop_\G = \sum_{j = 1}^mZ_j^2$
 which lifts $\LL$ through the projection of $\R^n\times \R^p$ onto $\R^n$.
 As usual, coordinates $(x,\xi)$ are fixed on $\R^n\times \R^p$. We know that
  $d_\lambda$ takes the form
  \begin{equation} \label{sec.two_2:eq_DstardeltaX}
   d_\lambda(x,\xi) =  (\delta_\lambda(x),\delta^*_\lambda (\xi) ),
  \end{equation}
  where $\delta^*_\lambda $ is the dilation  {on $\R^p$}
  in \eqref{delaprodottosa}.
  Three homogenous dimensions naturally arise:
  \begin{itemize}
    \item[-]
     that of $(\R^n,\dela)$, namely $q:= \sum_{j = 1}^n\sigma_j$;

    \item[-]
    that of $(\R^p,\delta^*_\lambda)$, namely $q^*:= \sum_{j = 1}^n \sigma^*_j$;

    \item[-]
    that of $(\R^N,d_\lambda)$, namely $Q=q+q^*$.
    \end{itemize}
   Let us now assume that the  {$\delta_\lambda$-dimension of $\R^n$}
  is greater than $2$:
\begin{equation} \label{sec.two_2:def_Qdelta}
  \textstyle  q= \sum_{j = 1}^n\sigma_j > 2.
 \end{equation}
 In the sequel, we consider
 the homogeneous norm on $\G$ (in the sense \cite[Def. 5.1.1]{BLU})
\begin{equation} \label{sec.two_2:eq_defcanonicalh}
   h(x,\xi) :=
   \sum_{j = 1}^n|x_j|^{\displaystyle {1}/{\sigma_j}} + \sum_{k = 1}^p |\xi_k|^{\displaystyle{1}/{\sigma^*_k}}.
\end{equation}
\begin{rem}\label{rem.sugamma}
  Under condition \eqref{sec.two_2:def_Qdelta}, the  {homogeneous dimension} $Q$
  of the Carnot group $\G$ is also greater than $2$,
  so that the following notable result holds true
  (see \cite[Theorem 2.1]{Folland}):

  \emph{The sub-Laplacian $\Lop_\G$ admits a unique
  {fundamental solution}
   with pole at $0$, that is,
   a function
   $\gamma_\G: \RN\to\R$
   satisfying the following properties:\medskip
   \begin{itemize}
    \item $\gamma_\G \in C^\infty(\RN\setminus\{0\},\R)$
    and $\gamma_\G > 0$ on $\RN\setminus\{0\}$;
    \item $\gamma_\G \in \Lloc{\RN}$ and $\gamma_\G$ vanishes at infinity;
    \item $\Lop_\G \gamma_\G = -\mathrm{Dir}_0$ in the weak sense
    of distributions;
    \item $\gamma_\G$ is $d_\lambda$-homogeneous of degree $2-Q$;
    \item if $h$ is as in \eqref{sec.two_2:eq_defcanonicalh}, there exists
    a (group) constant $\mathbf{c} > 0$ such that
\begin{equation} \label{sec.two_2:eq_estimateG0}
 \mathbf{c}^{-1}\,h^{2-Q} \leq \gamma_\G \leq  \mathbf{c}\,h^{2-Q}\quad \text{on $\G\setminus\{0\}$.}
\end{equation}
  \end{itemize}}
  \noindent One can then obtain the fundamental solution $\Gamma_\G$ of $\LL_\G$  under group-convolution:
\begin{equation} \label{sec.two_2:eq_defGLopG}
   \Gamma_\G(x,\xi;y,\eta)=\gamma_\G((x,\xi)^{-1}\star(y,\eta))\qquad \forall\,\,(x,\xi)\neq (y,\eta).
\end{equation}
 Occasionally, by an abuse of notation, we may denote $\gamma_\G$ by $\Gamma_\G$ as well.
\end{rem}
  Our main goal to show that $\Gamma_\G$
  satisfies the integrability assumptions Theorem
  \ref{sec.one:thm_Main} (plus the other good properties in
  Prop.\,\ref{sec.one:prop_contVanishGamma}). 
  As a consequence, since we proved in
  Theorem \ref{sec.two_1:thm_regliftLopX} that
  $\LL_\G$ is a saturable Lifting of $\LL$, then
  $\LL$ admits a fundamental solution
  obtained by a saturation of $\Gamma_\G$.
  This will prove Theorem \ref{thm.riassuntivoooo}.\medskip

 Due to its role in the saturation formula
 \eqref{sec.one:mainThm_defGamma},
 we study some properties of the map
   \begin{equation} \label{sec.two_2:eq_convmap}
    F: \R^n\times\RN\to\RN, \quad F (x,y,\eta) :=
    (x,0)^{-1}\star(y,\eta).
   \end{equation}
   First of all we observe that, since the family of dilations
   $\{d_\lambda\}_{\lambda > 0}$ forms a one-parameter group
   of automorphisms of $\G$, for every $x\in\R^n$
   and every $(y,\eta)\in \RN$ we have
   \begin{align*}
     F (\delta_\lambda &(x), d_\lambda(y,\eta) )=
     (\delta_\lambda(x),0 )^{-1}\star d_\lambda(y,\eta)
     \\
     & \stackrel{\eqref{sec.two_2:eq_DstardeltaX}}{=}
      (d_\lambda(x,0) )^{-1}\star d_\lambda(y,\eta)
     =
     d_\lambda ((x,0)^{-1}\star(y,\eta) )
     = d_\lambda (F (x, (y,\eta) ) );
   \end{align*}
   hence, if we consider the family of dilations
   $\{\widetilde{D}_\lambda\}_{\lambda > 0}$
   on $\R^n\times\RN$ given by
   $$\widetilde{D}_\lambda:\R^n\times\RN\to\R^n\times\RN,
   \quad \widetilde{D}_\lambda (x,y,\eta)
   =  (\delta_\lambda(x),d_\lambda(y,\eta) ),$$
   then the components of $F$, say
   $$F_1,\ldots,F_n,\quad F_{n+1},\ldots,F_N,$$
   are $\widetilde{D}_\lambda$-homogeneous of degrees, respectively,
   $$\sigma_1,\ldots,\sigma_n,\quad
   \sigma^*_1,\ldots,\sigma^*_p.$$
   On the other hand, if we take $x = 0$,
   we get $F (0,(y,\eta) )= (y,\eta)$,
   whilst $F (x,(x,0) )= (0,0)$
   (since the origin is the neutral element of $\G$).
   By all these facts, we deduce that
   the components of $F$ are $\widetilde{D}_\lambda$-homogeneous polynomials,
   and that, for every $x\in\R^n$ and every
   $(y,\eta)\in\RN$,
   they take the form
   \begin{equation} \label{sec.two_2:eq_expressionconvmap}
    \begin{split}
     F_1 (x, y,\eta) ) & = y_1-x_1, \\
     F_i (x,y,\eta) & = y_i-x_i+ p_i (x,y,\eta)  \qquad (i = 2,\ldots,n), \\
     F_{n+k} (x,y,\eta) & = \eta_k+ q_k (x,y,\eta),\qquad  (k = 1,\ldots,p),
    \end{split}
   \end{equation}
  where $p_i$ and $q_k$ are
  $\widetilde{D}_\lambda$-homogeneous polynomials
  of degrees $\sigma_i$ and $\sigma^*_k$, respectively, and
\begin{itemize}
  \item
  $p_i$ only depends on those variables
  $x_h,y_h$ and $\eta_j$ such that
  $\sigma_h,\sigma^*_j< \sigma_i$;

  \item $q_k$ only depends on those variables
  $x_h,y_h$ and $\eta_j$ such that
  $\sigma_h,\sigma^*_j< \sigma^*_k$;

  \item $p_i (0,y,\eta) = q_k (0,y,\eta ) = 0$,
  for every $(y,\eta)\in\RN$.
\end{itemize}
 \begin{rem}\label{sec.two_2:remChangeF}
 Let $x,y \in \R^n$ be fixed. Since the polynomial
   $q_1$  {does not depend} on $\eta_1,\ldots,\eta_p$ and since,
   for every $k \in\{2,\ldots,p\}$, the polynomial
   $q_k$ only depends on $\eta_1,\ldots,\eta_{k-1}$, we see that
   \begin{equation} \label{sec.two_2:def_mapPsixy}
    \Psi_{x,y}: \R^p\longto\R^p, \quad \Psi_{x,y}(\eta) :=  \Big(F_{n+1} (x,y,\eta),
    \ldots,F_N (x,y,\eta) \Big),
   \end{equation}
   defines a  {$C^\infty$-diffeomorphism}, with polynomial components.
   Hence, in particular, we have
   \begin{equation} \label{sec.two_2:Psiexplodes}
    \lim_{\|\eta\|\to\infty}\Psi_{x,y}(\eta) = \infty.
   \end{equation}
   Furthermore, by \eqref{sec.two_2:eq_expressionconvmap}, we get
   \begin{equation} \label{sec.two_2:eq_Jacobianone}
    \det (\mathcal{J}_{\displaystyle \Psi_{x,y}}(\eta) ) = 1,
    \quad \text{for every $\eta \in \R^p$}.
   \end{equation}
\end{rem}
 Summing up, by
 \eqref{sec.two_2:eq_estimateG0} and
 \eqref{sec.two_2:eq_defGLopG},
 we obtain (whenever
 $(y,\eta)\neq (x,0)$)
 \begin{equation} \label{sec.two_2:eq_mainEstimG}
   \mathbf{c}^{-1}\, K^{2-Q} (x,y,\eta)
   \leq \Gamma_\G (x,0;y,\eta) \leq
   \mathbf{c}\, K^{2-Q} (x,y,\eta),
  \end{equation}
  where we have set
  $$K (x,y,\eta) := h((x,0)^{-1}\star(y,\eta) ), \quad \text{with $h$ as in
  \eqref{sec.two_2:eq_defcanonicalh}}.$$
  Taking into account  \eqref{sec.two_2:eq_defcanonicalh}
  and \eqref{sec.two_2:eq_expressionconvmap}, a more explicit
  expression for $K$ is
\begin{equation} \label{sec.two_2:eq_defKernelK}
  \begin{split}
   K (x,y,\eta) & =
   \sum_{i = 1}^n \Big|F_i (x,y,\eta) \Big|^{\displaystyle{1}/{\sigma_i}} +
     \sum_{k = 1}^p \Big|F_{n+k} (x,y,\eta) \Big|^{\displaystyle {1}/{\sigma^*_k}}
     \\
     & =
    |y_1-x_1 | +
   \sum_{i = 2}^n \Big|y_i-x_i+p_i (x,y,\eta ) \Big|^{\displaystyle {1}/{\sigma_i}} +
   \sum_{k = 1}^p
    \Big|\eta_k+q_k (x,y,\eta ) \Big|^{\displaystyle {1}/{\sigma^*_k}}.
   \end{split}
  \end{equation}
  Thanks to \eqref{sec.two_2:eq_mainEstimG},
  we are now able to prove the following crucial result:
\begin{thm}\label{sec.two_2:thmassumptionI}
 Suppose that \eqref{sec.two_2:def_Qdelta} holds true.
 Then the fundamental solution
 $\Gamma_\G$ of $\Lop_\G$ satisfies
 assumptions \emph{(i)} and \emph{(ii)}
 in Theorem \ref{sec.one:thm_Main}.
 As a consequence,
\begin{equation*}
  \Gamma(x;y) = \int_{\R^p} \Gamma_\G \big(x,0;y,\eta\big)\,\d\eta\quad (x\neq y)
\end{equation*}
 is a fundamental solution for $\LL$.
 Moreover, if $h$
 is as in \eqref{sec.two_2:eq_defcanonicalh},
 one has global estimates
   $$\mathbf{c}^{-1}\,\int_{\R^p} h^{2-Q}((x,0)^{-1}\star(y,\eta) )\,\d \eta
   \leq \Gamma(x;y) \leq \mathbf{c}\,
   \int_{\R^p} h^{2-Q}((x,0)^{-1}\star(y,\eta) )\,\d \eta,$$
   holding true for every $x,y \in \R^n$ with $x\neq y$,
   where $\mathbf{c}$ is a constant (only depending on $\G$),
   $Q$ is the homogeneous dimension of $\G$,
   and $\star$ is the group law of $\G$.
\end{thm}
  \begin{proof}
   First we prove condition (i).
   We need to prove \eqref{sec.one:eq_sumtildeGamma_1}
   when $\widetilde{\Gamma}$ is $\Gamma_\G$;
   due to \eqref{sec.two_2:eq_mainEstimG}, we need to prove that,
   for fixed $x\neq y$ in $\R^n$, we have
\begin{equation}\label{sommabiliii}
  \eta \mapsto  K^{2-Q} (x,y,\eta)\quad \text{belongs to}\quad L^1(\R^p).
\end{equation}
  We perform the change of variable $\eta = \Psi_{x,y}^{-1}(u)$
  introduced in Remark \ref{sec.two_2:remChangeF}:
   \begin{align*}
    \int_{\R^p} K^{2-Q} (x,y,\eta)\,\d\eta
    &=
    \int_{\R^p} K^{2-Q} (x,y,\Psi_{x,y}^{-1}(u))\cdot
     |\det (\mathcal{J}_{\Psi^{-1}_{x,y}}(u) ) |\,\d u \\
    & \stackrel{\eqref{sec.two_2:eq_Jacobianone}}{=}
    \int_{\R^p} K^{2-Q} (x,y,\Psi_{x,y}^{-1}(u) )\,\d u.
   \end{align*}
   We now observe that, since $x\neq y$, the function
    $u\mapsto  K^{2-Q} (x,y,\Psi_{x,y}^{-1}(u) )$ is
    {continuous}
   on $\R^p$, hence it is integrable  {on every compact subset of $\R^p$}.
   In fact, $ K (x,y,\Psi_{x,y}^{-1}(u) ) = 0$ iff
   $$(x,0)^{-1}\star  (y,\Psi_{x,y}^{-1}(u) ) = 0,$$
   which necessarily implies $x = y$.
   Thus, if we consider the homogeneous norm $N$ in \eqref{sec.two_1:eq_defnormN},
  \eqref{sommabiliii} will follow if we show that
   $$
    \int_{\{N(u)\geq 1\}}\!\!\!\!
     K^{2-Q} (x,y,\Psi_{x,y}^{-1}(u) )\,\d u < \infty.
   $$
    By exploiting the expression
   of $ K$ given in \eqref{sec.two_2:eq_defKernelK}
   and the definition of $\Psi_{x,y}$, we infer
   \begin{align*}
     K& (x,y,\Psi_{x,y}^{-1}(u) )
    \stackrel{\eqref{sec.two_2:eq_defKernelK}}{\geq} \sum_{k = 1}^p
    |F_{n+k} (x,y,\Psi_{x,y}^{-1}(u))  |^{{1}/{\sigma^*_k}}   \\
   & \stackrel{\eqref{sec.two_2:def_mapPsixy}}{=}
   \sum_{k = 1}^p
    |\Psi_{x,y} (\Psi_{x,y}^{-1}(u) ) |^{{1}/{\sigma^*_k}}
   = \sum_{k = 1}^p|u_k|^{{1}/{\sigma^*_k}} {=} N(u).
   \end{align*}
  Therefore, we are left to show that
   \begin{equation} \label{sec.two_2:eq_toProvefinal}
  \int_{\{N(u)\geq 1\}}N^{2-Q}(u)\,\d u < \infty.
   \end{equation}
   In proving \eqref{sec.two_2:eq_toProvefinal},
   we use a typical argument on diadic annuli (modeled on the homogeneous norm $N$):
   setting, for $j \in \mathbb{N}$, $C_j := \{u \in \R^p: 2^{j-1}\leq N(u) < 2^{j}\}$,
   then (see \eqref{delaprodottosa} for the definition of $\delta^*_\lambda$)
   \begin{align*}
    \int_{\{N(u)\geq 1\}}&N^{2-Q}(u)\,\d u
    = \sum_{j = 1}^{\infty}\int_{C_j}N^{2-Q}(u)\,\d u \quad
     (\text{change of variable $u = \delta^*_{2^j}(\eta)$} ) \\
    & \,\,\,= \sum_{j = 1}^\infty (2^j )^{q^*}\,
    \int_{\delta^*_{2^j}(C_j)}\,N^{2-Q}
     (\delta^*_{2^j}(\eta) )\,\d \eta \\
    & \,\,\,= \left(\int_{\{1/2\leq N(\eta)\leq 1\}}\!\!\!
    N^{2-Q}(\eta)\,\d \eta\right)\,
    \sum_{j = 1}^\infty  (2^j )^{2-Q+q^*}<\infty,
    \end{align*}
   since $2-Q+q^*=2-q>0$  by \eqref{sec.two_2:def_Qdelta}.
   This ends the proof of (i).\medskip

   Finally we prove (ii) of Theorem \ref{sec.one:thm_Main}.
   We need to prove \eqref{sec.one:eq_sumtildeGamma_2}
   when $\widetilde{\Gamma}$ is $\Gamma_\G$.
  If $x\in \R^n$ is fixed and $K\subset \R^n$
  is compact, we perform the change of variable
 $ (u,v)=(y,\Psi_{x,y}(\eta))$ and we get (arguing as in (i)
 to recognize that this substitution has Jacobian determinant $\equiv 1$)
\begin{align*}
    &\int_{K\times\R^p}
    \Gamma_\G(x,0;y,\eta)\,\d y\,\d \eta
    =
    \int_{K\times\R^p}
    \Gamma_\G (x,0;u,\Psi_{x,u}^{-1}(v))\,\d u\,\d v \\
    &=\int_{K\times\{N(v)\leq 1\}}\{\cdots\}\,\,\d u\,\d v+
  \int_{K\times\{N(v)>1\}}\{\cdots\}\,\,\d u\,\d v =:\mathrm{I}+\mathrm{II},
\end{align*}
 where $N$ is as above. Clearly I is finite since we integrate a continuous function
 on a compact set.
 As for II, we use \eqref{sec.two_2:eq_mainEstimG} and we
 have to prove the finiteness of the following integral:
   \begin{align*}
  &\int_{K\times\{N(v)> 1\}}
     K^{2-Q} (x,u,\Psi_{x,u}^{-1}(v) )\,\d u\,\d v \\
    & \stackrel{\eqref{sec.two_2:eq_defKernelK}}{\leq}
    \int_{K\times\{N(v)> 1\}}
    \left(\sum_{k = 1}^p
     |F_{n+k} (x,u,\Psi_{x,u}^{-1}(v) )
     |^{{1}/{\sigma^*_k}}\right)^{2-Q}\,\d u\,\d v \\
     & \stackrel{\eqref{sec.two_2:def_mapPsixy}}{=}
  \int_{K\times\{N(v)> 1\}}
    \left(\sum_{k = 1}^p |v_k|^{{1}/{\sigma^*_k}}\right)^{2-Q}\,\d u\,\d v
    = \mathbf{c}\,\int_{K\times\{N(v)>1\}}N^{2-Q}(v)\,\d u\,\d v.
   \end{align*}
 The finiteness of the last integral follows by the same argument
 as in the previous part of the proof (and the fact that $K$ is compact).
  \end{proof}
\begin{prop}\label{sec.two_2:propGammaHom}
 In the assumption and notation of Theorem
 \ref{sec.two_2:thmassumptionI},
 the function $\Gamma$ fulfils the following
 (joint) $\dela$-homogeneity property:
 \begin{equation} \label{sec.two_2:eq_GammaHomogeneous}
    \Gamma (\delta_\lambda(x);\delta_\lambda(y) ) = \lambda^{2-q}\,
    \Gamma(x;y), \quad \text{for every $x,y \in \R^n$ with $x \neq y$ and $\lambda>0$}.
   \end{equation}
 Furthermore, it is continuous out of the diagonal of $\R^n\times\R^n$, and it is symmetric:
 \begin{equation} \label{sec.two_2:eq_GammaHomogeneousSYMM}
\Gamma(x;y)=\Gamma(y;x) \qquad \text{for every $x,y\in\R^n$ with $x\neq y$.}
\end{equation}
 Finally, for every fixed $x \in \R^n$, we have the following properties:\medskip
    \begin{itemize}
    \item[(i)]
   $\Gamma(x;\cdot)=\Gamma(\cdot;x)$ is smooth and $\LL$-harmonic on $\R^n\setminus\{x\}$;

    \item[(ii)]
    $\Gamma(x;\cdot)=\Gamma(\cdot;x)$ vanishes at infinity (uniformly for $x$ in compact sets);

   \item[(iii)]
    $\Gamma$ is locally integrable on $\R^n\times \R^n$
    and $C^\infty$ out of the diagonal of $\R^n\times \R^n$.
   \end{itemize}
    \end{prop}
      \begin{proof}
   Let $\lambda > 0$ and let $x,y \in \R^n$ be distinct. We have
   \begin{align*}
    \Gamma (\delta_\lambda(x);&\delta_\lambda(y) )
 =   \int_{\R^p}\Gamma_\G(d_\lambda(x,0);\delta_\lambda(y),\eta)\,
    \d \eta.
   \end{align*}
   By the substitution
   $\eta = \delta^*_{\lambda}(u)$, and by the $d_\lambda$-homogeneity of degree $2-Q$ of $\Gamma_\G$, we obtain
   \begin{align*}
    \Gamma(\delta_\lambda(x);\delta_\lambda(y) )
    &= \lambda^{q^*}\,
    \int_{\R^p}
    \Gamma_\G(d_\lambda(x,0);d_\lambda(y,u))\,
    \d u \\
    &= \lambda^{2-Q+q^*}\,\int_{\R^p}
    \Gamma_\G(x,0;y,u)\,\d u=  \lambda^{2-q}\,
    \Gamma(x;y),
   \end{align*}
  since $Q=q+q^*$. This gives \eqref{sec.two_2:eq_GammaHomogeneous}.
  The proofs of \eqref{sec.two_2:eq_GammaHomogeneousSYMM} and (i)-to-(iii)
  are technical and long; for this reason we postpone them to the Appendix, Section \ref{sec:Appendix}.
 \end{proof}
  \section{Examples}\label{sec:examples}
 This last section is devoted to present some explicit examples of linear homogeneous
 PDOs to which our theory applies.
\begin{ex}[\textbf{Grushin operator on $\R^2$}] \label{sec.th:ex_Grushin}
   Let us consider  the vector fields  on $\R^2$
\begin{equation*}
    X_1 = \de_{x_1}, \quad X_2 = x_1\,\de_{x_2}.
\end{equation*}
   It is readily seen that $X_1$ and $X_2$ are
   homogeneous of degree $1$ with respect to the dilations
 $$\delta_\lambda(x_1,x_2)=(\lambda x_1,\lambda^2 x_2).$$
 Obviously,  assumptions (H1) and (H2)
  of Section \ref{sec.two} are satisfied and Theorem
  \ref{sec.two_1:thmLiftMain} can be applied; the relevant
   Carnot group is $\G = (\R^3,\star,d_\lambda)$ with
\begin{equation*}
    d_\lambda(x_1,x_2,\xi) = (\lambda x_1,\lambda^2 x_2, \lambda\,\xi),\quad Q=4,
\end{equation*}
  while the composition law is
  \begin{equation*}
   (x_1,x_2,\xi)\star(y_1,y_2,\eta) := (x_1+y_1, x_2+y_2+x_1\eta, \xi+\eta).
  \end{equation*}
  Furthermore, the vector fields $Z_1,Z_2$ lifting $X_1$ and $X_2$ are
  \begin{equation} \label{sec.th:exGrushinZs}
   Z_1 = \de_{x_1}, \qquad Z_2 = x_1\,\de_{x_2}+\de_\xi.
  \end{equation}
  The operator $\Lop= X_1^2+X_2^2$ lifts to
  the sub-Laplacian $\Lop_\G = Z_1^2+Z_2^2$.
  The latter is (modulo a change of variable)
  the Kohn-Laplacian on the first Heisenberg group, whence its fundamental solution
  with pole at the origin is
  $$\Gamma_\G(x,\xi) = c\, \Big((x_1^2+\xi^2)^2+16\,(x_2-1/2\,x_1\xi)^2 \Big)^{-1/2},
  \quad (x,\xi)\neq (0,0),$$
  where $c > 0$ is a suitable constant.
  According to Theorem \ref{sec.two_2:thmassumptionI}, the function
\begin{equation} \label{sec.th:GammaGrushin}
  \Gamma(x_1,x_2;y_1,y_2) = c\,\int_{\R}\frac{\d\eta}
  {\sqrt{ ((x_1-y_1)^2+\eta^2 )^2+ 4\,
    (2\,x_2 - 2\,y_2 + \eta\,(x_1+y_1) )^2}},
 \end{equation}
   is the  unique fundamental solution for the Grushin operator $\Lop$ vanishing at infinity.
   From \eqref{sec.th:GammaGrushin} we also derive that, for every
   $x \in \R^2$, the function $\Gamma(x;\cdot)$ has a pole at
   $x$: in fact (see Proposition \ref{sec.one:prop_poleGamma})
   $$\liminf_{y\to x}\Gamma(x,y) \geq c\,\int_{\R}
   \frac{\d\eta}{\sqrt{\eta^4 + 16\,x_1^2\eta^2}} = \infty.$$
   Finally, the integral in \eqref{sec.th:GammaGrushin} can be expressed in terms of
    Elliptic Functions: more precisely, we have
   \begin{equation} \label{sec.th:GammaGrushinExplicit}
    \Gamma(x;y) = \frac{c\,\sqrt{2}}{\sqrt[4]{(x_1^2+y_1^2)^2+4\,(x_2-y_2)^2}}\cdot
    \mathrm{K}\left(\frac{1}{2}+
    \frac{x_1y_1}{\sqrt[4]{(x_1^2+y_1^2)^2+4\,(x_2-y_2)^2}}\right),
   \end{equation}
   where $\mathrm{K}$ denotes the complete elliptic integral of the first
   kind, that is,
   $$\mathrm{K}(m) := \int_0^{\pi/2} (1-m\,\sin^2(t) )^{-1/2}\,\d t,
   \quad \text{for $-1 < m < 1$}.$$
   This gives back a formula already obtained by Greiner \cite{Greiner}
   (see also Beals, Gaveau, Greiner \cite{BealsGaveauGreiner1, BealsGaveauGreiner3};
   Beals, Gaveau, Greiner, Kannai \cite{BealsGaveauGreinerKannai2}; Bauer, Furutani, Iwasaki \cite{BauerFurutaniIwasaki}).
 \end{ex}
 \begin{ex}[\textbf{Another Grushin-type operator}] \label{sec.th:exGrushinII}
  Let us consider the
   vector fields on $\R^2$
   \begin{equation*}
    X_1 = \de_{x_1}, \quad X_2 = x_1^2\,\de_{x_2}.
   \end{equation*}
   They are homogeneous of degree $1$ with respect to the dilations
 $$\delta_\lambda(x_1,x_2)= (\lambda x_1,\lambda^3 x_2).$$
 The Lie algebra $\mathfrak{a}$ generated by $X_1,X_2$ is $4$-dimensional
 and the dimensions of the layers of the stratification
 as in \eqref{sec.two_1:astratified} are, respectively, $2,1,1$.
 Assumptions (H1) and (H2)
 of Section \ref{sec.two} are satisfied and Theorem
 \ref{sec.two_1:thmLiftMain} provides us with the
 Carnot group  $\G = (\R^4,\star,d_\lambda)$ where
 \begin{equation*}
    d_\lambda(x_1,x_2,\xi_1,\xi_2)
    = (\lambda x_1,\lambda^3 x_2, \lambda\,\xi_1,\lambda^2\,\xi_2),
  \end{equation*}
 and the composition law $(x_1,x_2,\xi_1,\xi_2)\star(y_1,y_2,\eta_1,\eta_2)$ is
 \begin{equation*}
  \begin{split}
    (x_1+y_1, x_2 + y_2 + x_1(x_1+y_1)\eta_1+2x_1\eta_2,
   \xi_1+\eta_1, \xi_2+\eta_2+1/2(x_1\eta_1-y_1\xi_1) ).
   \end{split}
  \end{equation*}
  The vector fields $Z_1,Z_2$ lifting $X_1$ and $X_2$ are
 \begin{equation*}
   Z_1 = \de_{x_1}-\frac{\xi_1}{2}\,\de_{\xi_2},
   \qquad Z_2 = x_1^2\,\de_{x_2}+\de_{\xi_1}+\frac{x_1}{2}\,\de_{\xi_2}.
\end{equation*}
  The sub-Laplacian $\Lop_\G = Z_1^2+Z_2^2$ lifts
  the (3-step) Grushin-type operator $\Lop= X_1^2+X_2^2$.
 If $\Gamma_\G$ is the fundamental solution for $\Lop_\G$,
 by Theorem \ref{sec.two_2:thmassumptionI} the function
\begin{equation*}
   \begin{split}
   \Gamma(x;y) & = \int_{\R^2}\Gamma_\G ((x,0)^{-1}\star(y,\eta) )\,\d\eta\\
   &= \int_{\R^2}\Gamma_\G \Big(y_1-x_1,y_2-x_2 + x_1\eta_1(x_1-y_1)-2\,x_1\eta_2,
   \eta_1,\eta_2-\tfrac{1}{2}x_1\eta_1 \Big)\,\d\eta_1\d\eta_2,
   \end{split}
  \end{equation*}
  is the  unique fundamental solution for $\Lop$ vanishing at infinity.

  Furthermore, from Theorem \ref{sec.two_2:thmassumptionI}
  we derive that $\Gamma(x;y)$ is bounded from above and from below (up to
  two structural constants) by
\begin{equation*}
  \int_{\R^2} K^{-5} (x,y,\eta)\,\d \eta_1\d\eta_2,
\end{equation*}
   where the function $ K$ is
   \begin{align*}
    &K (x,y,\eta) =
    |y_1-x_1 |+ |
   y_2-x_2 + x_1\eta_1(x_1-y_1)
   -2\,x_1\eta_2 |^{1/3}+ |\eta_1 |+
    |\eta_2-\tfrac{1}{2}x_1\eta_1 |^{1/2}.
   \end{align*}
   In this case we are able to deduce that,
   for every fixed $x\in\R^2$, the function $\Gamma(x;\cdot)$ has a pole at $x$
   (see Proposition \ref{sec.one:prop_poleGamma}): indeed, for some constant $\mathbf{c}>0$
   \begin{align*}
    \liminf_{y\to x}\Gamma&(x,y) \geq\mathbf{c}^{-1}\,\int_{\R^2} \Big( |2x_1\eta_2 |^{1/3}
   + |\eta_1 |+
    |\eta_2-1/2x_1\eta_1 |^{1/2} \Big)^{-5}\,\d\eta = \infty.
   \end{align*}
 \end{ex}
 \begin{ex}[\textbf{An Engel-type operator}] \label{sec.th:ex_Engel}
  Let us now consider the vector fields on $\R^3$
   \begin{equation*}
    X_1 = \de_{x_1}, \quad X_2 = x_1\,\de_{x_2}+x_1^2\,\de_{x_3}.
   \end{equation*}
  They are homogeneous of degree $1$ with respect to
\begin{equation*}
\delta_\lambda(x_1,x_2,x_3) :=
     (\lambda x_1,\lambda^2 x_2,\lambda^3 x_3).
    \end{equation*}
 The Lie algebra $\mathfrak{a}$ generated by $X_1,X_2$ is $4$-dimensional
 and the dimensions of the layers of the stratification
 as in \eqref{sec.two_1:astratified} are, respectively, $2,1,1$.
 Assumptions (H1) and (H2)
 of Section \ref{sec.two} are satisfied and Theorem
 \ref{sec.two_1:thmLiftMain} provides us with the
 Carnot group  $\G = (\R^4,\star,d_\lambda)$ where
 \begin{equation*}
    d_\lambda(x_1,x_2,x_3,\xi)
    = (\lambda x_1,\lambda^2 x_2, \lambda^3 x_3,\lambda\,\xi),
  \end{equation*}
  while the composition $(x_1,x_2,x_3,\xi)\star(y_1,y_2,y_3,\eta)$ is
  \begin{equation*}
   (x_1+y_1, x_2+y_2+x_1\eta,
   x_3+y_3+2x_1y_2+x_1^2\eta, \xi+\eta).
  \end{equation*}
 $\G$ is isomorphic to the Engel group on $\R^4$.
  The vector fields $Z_1,Z_2$ lifting $X_1,X_2$ are
  \begin{equation*}
   Z_1 = \de_{x_1}, \qquad Z_2 = x_1\,\de_{x_2}+x_1^2\,\de_{x_3}+\de_\xi.
  \end{equation*}
 The sub-Laplacian $\Lop_\G = Z_1^2+Z_2^2$ lifts $\Lop= X_1^2+X_2^2$. If
 $\Gamma_\G$ is the fundamental solution of $\Lop_\G$,
 by Theorem \ref{sec.two_2:thmassumptionI} the function
\begin{equation*}
   \begin{split}
   \Gamma(x_1,x_2,x_3;y_1,y_2,y_3) &
   = \int_{\R} \Gamma_\G\Big(y_1-x_1,y_2-x_2 - x_1\eta,
   y_3-x_3+2x_1(x_2-y_2)+x_1^2\eta, \eta\Big)\,\d\eta
   \end{split}
 \end{equation*}
  is the  unique fundamental solution for $\Lop$ vanishing at infinity.

  Furthermore, from Theorem \ref{sec.two_2:thmassumptionI}
  we derive that $\Gamma(x;y)$ is bounded from above and from below (up to
  two structural constants) by
\begin{equation*}
   \int_{\R}
   \Big\{|y_1-x_1 |+
    | y_2-x_2 - x_1\eta |^{1/2}+ \Big|y_3-x_3+2x_1(x_2-y_2)+x_1^2\eta \Big|^{1/3}+
    |\eta |\Big\}^{-5}\,\d\eta.
\end{equation*}
   In this case we are able to deduce that,
   for every fixed $x\in\R^3$, the function $\Gamma(x;\cdot)$ has a pole at $x$
   (see Proposition \ref{sec.one:prop_poleGamma}): indeed, for some constant $\mathbf{c}>0$
\begin{align*}
   & \liminf_{y\to x}\Gamma(x;y) \geq
   \mathbf{c}^{-1}\,\int_{\R} ( |x_1\eta |^{1/2}
   + |x_1^2 \eta |^{1/2}+
    |\eta | )^{-5}\,\d\eta = \infty.
\end{align*}
\end{ex}
 \section{Appendix: Further qualitative properties of $\Gamma$}\label{sec:Appendix}
 We tacitly follow all the notation of the previous sections.
  \begin{prop} \label{sec.two_2:eq_GammaVanish}
 For every fixed $x \in \R^n$ one has
   \begin{itemize}
    \item[$\mathrm{(i)}$] $\Gamma(x;\cdot)$ is continuous on $\R^n\setminus\{x\}$;
    \item[$\mathrm{(ii)}$] $\Gamma(x;\cdot)$ vanishes at infinity.
   \end{itemize}
  \end{prop}
  \begin{proof}
   These properties of $\Gamma$ are inherited from those of $\Gamma_\G$:
   it suffices to apply Dominated Convergence in the integral
   defining $\Gamma$, via the change of variable
    $\eta = \Psi_{x,y}^{-1}(u)$, with $u\in \R^p$.
   \end{proof}
  \begin{cor}\label{sec.two_2:corSmoothGamma}
   If $x \in \R^n$, we have
    $\Gamma(x;\cdot)\in C^\infty(\R^n\setminus\{x\},\R)$ and
    $\Lop\Gamma(x;\cdot) = 0$ out of $x$.
  \end{cor}
  \begin{proof}
   It is a standard consequence of the
   $C^\infty$-hypoellipticity of $\LL$ (which is a H\"ormander operator) and of the continuity
   of $\Gamma(x;\cdot)$ out of $x$ (see Proposition \ref{sec.two_2:eq_GammaVanish}).
  \end{proof}
 \begin{lem}\label{app:lemIntegralG}
  The following properties hold true: \medskip

  \emph{(i)}\,\,the map $(x,y,\eta)\mapsto {\Gamma}_\G\big(x,0;y,\eta\big)$
  is locally integrable on $\R^n\times\RN$; \medskip

  \emph{(ii)}\,\,for every $y\in\R^n$, the map $(x,\eta)\mapsto {\Gamma}_\G\big(x,0;y,\eta\big)$
  is locally integrable on $\RN$.
 \end{lem}
 \begin{proof}
  (i) Let $K_1\subseteq\R^n$ and $K_2\subseteq\RN$
  be compact sets.
  By Fubini's Theorem and by
  the change of variable $(y,\eta) = (x,0)\star(z,\zeta)$, we get
  \begin{align*}
   \int_{K_1\times K_2}{\Gamma}_\G\big(x,0;y,\eta\big)\,\d x\,\d y\,\d\eta
   \stackrel{\eqref{sec.two_2:eq_defGLopG}}{=}
   \int_{K_1}\left(\int_{\tau_{x}^{-1}(K_2)}\gamma_\G(z,\zeta)
   \,\d z\,\d \zeta\right)\d x,
  \end{align*}
  where $\tau_{x}$ denotes the left-translation
  on $\G$ by $(x,0)$.
  We now observe that, for every $x\in K_1$, the set
  $\tau_{x}^{-1}(K_2)$ is in the compact set
  $H = \big(K_1\times\{0\}\big)^{-1}\star K_2$; hence, by recalling that
  $\gamma_\G$ is locally integrable on $\RN$, we obtain
  property (i).\medskip

  (ii) We fix a point $y\in\R^n$ and a compact set
  $K\subseteq \R^n$, and we set
  $$\mathrm{C}_y:\RN\longto\RN, \qquad \mathrm{C}_y(x,\eta) :=
  (x,0)^{-1}\star(y,\eta).$$
  It can be easily deduced from \eqref{sec.two_2:eq_expressionconvmap}
  that $\mathrm{C}_y$ is a $C^\infty$-diffeomorphism of
  $\RN$ onto itself; hence,
  by change of variable
  $(x,\eta) = \mathrm{C}_y^{-1}(z,\zeta)$, we get
  \begin{align*}
   \int_{K}{\Gamma}_\G\big(x,0;y,\eta\big)\,\d x\,\d\eta
 = \int_{\mathrm{C}_y(K)}{\gamma}_\G(z,\zeta)\,
   \det\big(\mathcal{J}_{\mathrm{C}_y}(z,\zeta)\big)\,\d z\,\d\zeta.
  \end{align*}
  Since $\mathrm{C}_y(K)$ is compact and since ${\Gamma}_0\in\Lloc{\R^N}$, we get (ii).
 \end{proof}
 \begin{prop}\label{app:propIntGamma}
  The function $\Gamma$ satisfies the following properties: \medskip

  \emph{(i)}\,\,$\Gamma\in \Lloc{\R^n\times\R^n}$; \medskip

  \emph{(ii)}\,\,for every $y \in\R^n$, we have $\Gamma(\cdot;y)\in\Lloc{\R^n}$.
 \end{prop}
 \begin{proof}
  (i) Let $K_1,K_2\subseteq\R^n$ be compact sets and let
   $\Phi$ be the map defined as follows:
   $$\Phi: \R^n\times\RN\longto\R^n\times\RN,
  \qquad \Phi(x,y,\eta) := \big(x,y,\Psi_{x,y}(\eta)\big).$$
  As pointed out in Remark
  \ref{sec.two_2:remChangeF},
  $\Psi_{x,y}$ is a smooth diffeomorphism
  of $\R^n$ onto itself, and the map
  $(x,y,\eta)\mapsto\Psi_{x,y}(\eta)$ is smooth on $\R^n\times\RN$; hence,
  $\Phi$ defines a smooth diffeomorphism
  and
  $$\text{$\det\big(\mathcal{J}_{\Phi}(x,y,\eta)\big)=1$
  for every $(x,y,\eta)\in\R^n\times\RN$.}$$
  From this, by Fubini's Theorem and
  by the change of variable
  $(x,y,\eta) = \Phi^{-1}(u,v,\nu)$, we get
  \begin{align*}
   \int_{K_1\times K_2}\Gamma(x;y)\,\d x\,\d y
   & = \int_{K_1\times K_2\times\R^p}{\Gamma}_\G\big(u,0;v,\Psi_{u,v}^{-1}(\nu)\big)\,
   \,\d u\,\d v\,\d\nu \\
   & = \int_{K_1\times K_2\times\{N(\nu)<1\}}\{\ldots\}  + \int_{K_1\times K_2\times\{N(\nu)\geq 1\}}\{\ldots\}   =:  \mathrm{I} +  \mathrm{II},
  \end{align*}
  where, as usual, we have set $N(\nu) = \sum_{k=1}^p|\nu_k|^{1/s_{j_k}}$.

  First of all we observe that, since the product
   $K_1\times K_2\times\{N(\nu)<1\}$ is bounded and since
   ${\Gamma}_\G$ is locally integrable
   (Lemma \ref{app:lemIntegralG}), then $ \mathrm{I}$ is finite.
   As for $ \mathrm{II}$ we notice that, from
   \eqref{sec.two_2:eq_mainEstimG} and the definition of $\Psi_{u,v}$, we get
    \begin{align*}
      \mathrm{II} &
     \stackrel{\eqref{sec.two_2:eq_mainEstimG}}{\leq}
      {c}\,\int_{K_1\times K_2\times\{N(v)\geq 1\}}
      {K}^{2-Q}\big(u,(v,\Psi_{u,v}^{-1}(\nu))\big)\,\d u\,\d v\,\d\nu \\
     & \stackrel{\eqref{sec.two_2:eq_defKernelK}}{\leq}
      {c}\,\int_{K_1\times K_2\times\{N(v)\geq 1\}}
    \left(\sum_{k = 1}^p
    \big|F_{n+k}\big(u,(v,\Psi_{u,v}^{-1}(\nu))\big)
    \big|^{\frac{1}{s_{j_k}}}\right)^{2-Q}\,\d u\,\d v\,\d\nu \\
    &\, \stackrel{\eqref{sec.two_2:def_mapPsixy}}{=}
  {c}\,\int_{K_1\times K_2\times\{N(\nu)\geq 1\}}N^{2-Q}(\nu)
    \,\d u\,\d v\,\d\nu.
    \end{align*}
   The finiteness of $\mathrm{II}$ follows from the integrability of $N^{2-Q}$ on $\{\nu\in\R^p:N(\nu)\geq 1\}$.\medskip

    (ii) Let $y\in\R^n$ and let $K\subseteq\R^n$ be a compact set.
    We consider the map
    $$\Phi_y:\RN\longto\RN, \qquad \Phi_y(x,\eta) :=
    \big(x,\Psi_{x,y}(\eta)\big).$$
    By arguing as above, one recognizes that $\det\big(\mathcal{J}_{\Phi_y}(x,\eta)\big) = 1$;
    therefore, by the change of variable $(x,\eta) = \Phi_y^{-1}(u,v)$, we get
    \begin{align*}
     \int_K\Gamma(x;y)\,\d x &= \int_{K\times\R^p}{\Gamma}_\G\big(u,0;y,\Psi_{u,y}^{-1}(v)\big)\,\d u\,\d v \\
    & = \int_{K\times\{N(v)<1\}}\{\ldots\} +
    \int_{K\times\{N(v)\geq 1\}}\{\ldots\}=:\mathrm{I}+ \mathrm{II}.
    \end{align*}
   Since $K\times\{N(v)<1\}$ is bounded, by Lemma \ref{app:lemIntegralG}-(ii) we infer that
   $ \mathrm{I}$ is finite; the finiteness of  $ \mathrm{II}$ can be proved as above.
 \end{proof}
  \begin{prop}\label{sec.two_2:propWellGammaBis}
   The following facts hold true: \medskip

   \emph{(i)}\,\,setting $\mathcal{O} := \{(x,y)\in\R^n\times\R^n:x \neq y\}$, then
   $\Gamma\in C(\mathcal{O},\R)$; \medskip

   \emph{(ii)}\,\,for every compact set $K\subseteq\R^n$, then
   $$\lim_{\|y\|\to\infty}\Gamma(x;y) = 0,
   \quad \text{\emph{uniformly} for $x\in K$};$$

   \emph{(iii)}\,\,for every fixed $y \in \R^n$, the function
   $\Gamma(\cdot;y)$ vanishes at infinity.
  \end{prop}
  \begin{proof}
  (i) follows by a standard argument of dominated convergence, by the use of
  the change of variable $\Psi_{x,y}$ and of the level sets of the homogeneous norm $N$;
  obviously, the continuity of $\Gamma_\G$ out of the diagonal of $\RN\times\RN$ is exploited.

  (ii) follows again by dominated convergence; we also invoke the estimate (valid for $K\subseteq\R^n$ compact, and uniformly for $y\in \R^n$)
  $$\sup_{x\in K}{\Gamma}_\G\big(x,0;y,\Psi_{x,y}^{-1}(u)\big)\leq N^{2-Q}(u),\qquad \text{for every $u\in\R^p$}.$$
 Note that
  $$\lim_{\|(y,u)\|\to\infty}
    (x,0)^{-1}\star(y,\Psi_{x,y}^{-1}(u))\to\infty \quad
  \text{uniformly for $x\in K$}.$$

  (iii) follows by arguing as in (ii), once one has observed that
  $$\lim_{\|(x,u)\|\to\infty}(x,0)^{-1}\star(y,\Psi_{x,y}^{-1}(u))
  = \infty.$$
 This ends the proof.
\end{proof}
  \begin{thm}[$\Gamma$ right-inverts $\Lop$]\label{app:thmGammaRInverts}
   For every $\varphi\in C_0^\infty(\R^n,\R)$, the function
   \begin{equation} \label{app:eq_defLambdaPhi}
    \Lambda_\varphi:\R^n\longto\R \qquad \Lambda_\varphi(y) :=
    \int_{\R^n}\Gamma(x;y)\,\varphi(x)\,\d x
   \end{equation}
   is well-defined and it satisfies the following properties: \medskip

   \emph{(i)}\,\,$\Lambda_\varphi\in \Lloc{\R^n}$; \medskip

   \emph{(ii)}\,\,$\Lambda_\varphi
   \in C\big(\R^n\setminus\mathrm{supp}(\varphi),\R\big)$
   and it vanishes at infinity; \medskip

   \emph{(iii)}\,\,$\Lop\big(\Lambda_\varphi\big) = -\varphi$ in the weak sense
   of distributions on $\R^n$.
  \end{thm}
  \begin{proof}
   It follows in a standard way, by using
   Propositions \ref{app:propIntGamma} and \ref{sec.two_2:propWellGammaBis}.
  \end{proof}
  \begin{cor} \label{app:corGammaRightInv}
   For every $\varphi \in C_0^{\infty}(\R^n,\R)$, one has
   \begin{equation} \label{app:eq_LambdaLphi}
    \Lambda_{\Lop\varphi}(y) = \int_{\R^n}\Gamma(x;y)\,\Lop\varphi(x)\,\d x
    = -\varphi(y), \quad \text{a.e.\,on $\R^n$}.
   \end{equation}
  \end{cor}
  \begin{proof}
   We know from Theorem
   \ref{app:thmGammaRInverts}
   that $\Lambda_{\Lop\varphi}\in\Lloc{\R^n}$ and that
   $\Lop(\Lambda_{\Lop\varphi})
   = -\Lop\varphi$ in the weak sense of distributions on $\R^n$;
   hence, since $\Lop$ is $C^\infty$-hypoelliptic, there exists
   $h\in C^\infty(\R^n,\R)$ such that
 $$\text{$\Lop h = -\Lop\varphi$ point-wise on $\R^n$ and $h = \Lambda_{\Lop\varphi}$ almost everywhere on $\R^n$.}$$
  Since $\Lambda_{\Lop\varphi}$
   is continuous outside $\mathrm{supp}(\varphi)$ (see Theorem \ref{app:thmGammaRInverts}), we have
   $h = \Lambda_{\Lop_\varphi}$ on $\R^n\setminus\mathrm{supp}(\varphi)$.
   As a consequence, the function $\Lambda_{\Lop_\varphi}$ vanishing at infinity,
   then the same is true of $h$.

   Let us now set $u := h + \varphi$. Obviously, we have
   $u\in C^\infty(\R^n,\R)$ and
   \begin{equation} \label{app:eq_LopuLopphi}
   \Lop u = \Lop h + \Lop\varphi = -\Lop\varphi+\Lop\varphi = 0,
  \quad \text{on $\R^n$};
  \end{equation}
  moreover, since $\varphi$ has compact support and $h$ vanishes at infinity,
  then
  \begin{equation} \label{app:eq_uvanishes}
  u(y)\to 0, \quad \text{as $\|y\|\to\infty$}.
  \end{equation}
  By \eqref{app:eq_LopuLopphi} and
  \eqref{app:eq_uvanishes}, we deduce from the Strong Maximum Principle
  for $\Lop$ that
  $u$ must vanish identically on $\R^n$, whence
 $h(y) = -\varphi(y)$ for every $y\in\R^n$.
  From this, by recalling that $h$ coincides almost everywhere with $\Lambda_\varphi$,
  we immediately obtain the desired identity
  \eqref{app:eq_LambdaLphi}, and the proof is complete.
  \end{proof}
  Thanks to all the results established so far, we can finally prove
  the symmetry of $\Gamma$.
  \begin{thm} [Symmetry of $\Gamma$] \label{app:thmSymGamma}
   The function $\Gamma$ is symmetric, that is,
   \begin{equation} \label{app:eq_SymGammaThm}
    \Gamma(x;y) = \Gamma(y;x), \quad \text{for every $x,y\in\R^n$ with $x\neq y$}.
   \end{equation}
  \end{thm}
  \begin{proof}
   We prove the existence of a measurable
   set $E\subseteq\R^n$, with $0$ Lebesgue measure and 
   \begin{equation} \label{app:eq_LGammaxStepI}
    \Lop\Gamma(\cdot;x) = -\mathrm{Dir}_x,
    \quad \text{for every $x\in\R^n\setminus E$}.
   \end{equation}
   To this end we observe that, the space $C^\infty_0(\R^n,\R)$ being separable
   (with the usual topology), there exists a countable dense
   set $\mathcal{F}\subseteq C_0^\infty(\R^n,\R)$; moreover,
   thanks to Corollary
   \ref{app:eq_LambdaLphi}, for every $\varphi\in \mathcal{F}$
  it is possible to find a measurable set $E(\varphi)$, with zero Lebesgue measure,
   such that
   $$\int_{\R^n}\Gamma(y;x)\,\Lop\varphi(y)\,\d y = -\varphi(x),
   \quad \text{for every $x\in\R^n\setminus E(\varphi)$}.$$
   We then set $E:=\bigcup_{\varphi\in\mathcal{F}}E(\varphi)$.
   Since $\mathcal{F}$ is countable and $E(\varphi)$ has measure $0$
   for every $\varphi\in\mathcal{F}$, we see that
   $E$ has measure $0$ as well; furthermore, for every $x\in\R^n\setminus E$, we have
   $$\int_{\R^n}\Gamma(y;x)\,\Lop\varphi(y)\,\d y = -\varphi(x),
   \quad \text{for every $\varphi\in\mathcal{F}$}.$$
   This proves that, for every $x\notin E$,
   the distribution $\Lop\Gamma(\cdot;x)$ coincides
   with $-\mathrm{Dir}_x$ on $\mathcal{F}$, so that, the latter being dense,
   we immediately obtain the claimed \eqref{app:eq_LGammaxStepI}.

   We now consider, for $x\notin E$,
   the function $u_x := \Gamma(x;\cdot)-\Gamma(\cdot;x)$.
   Obviously, $u_x\in\Lloc{\R^n}$ (since the same is true of
   $\Gamma(x;\cdot)$ and $\Gamma(\cdot;x)$) and, thanks to \eqref{app:eq_LGammaxStepI},
   we have
   $$\Lop u_x = \Lop\Gamma(x;\cdot)-\Lop\Gamma(\cdot;x)
   = -\mathrm{Dir}_x + \mathrm{Dir}_x = 0, \quad \text{in $\mathcal{D}'(\R^n)$};$$
   as a consequence, $\Lop$ being $C^\infty$-hypoelliptic,
   there exists $h_x\in C^\infty(\R^n,\R)$ such that
   $$\text{$\Lop h_x = 0$ on $\R^n$ and $h_x = u_x$ almost everywhere on $\R^n$.}$$
   In particular, as $u_x$ is continuous on $\R^n\setminus\{x\}$
   and it vanishes at infinity
   (see Proposition \ref{sec.two_2:propWellGammaBis}),
   $$\text{$h_x(y) = u_x(y) = \Gamma(x;y)-\Gamma(y;x)$,
  for every $y\in\R^n\setminus\{x\}$;  $h_x(y)\to 0$ as $\|y\|\to\infty$.}$$
  We deduce
  from the Strong Maximum Principle for $\Lop$ that
  $h_x\equiv 0$ on $\R^n$, whence
  \begin{equation} \label{app:eqGammaSymQuasi}
    \Gamma(x;y) = \Gamma(y;x), \quad \text{for every
  $x\notin E$ and every $y\in\R^n\setminus\{x\}$}.
  \end{equation}
   To complete the proof, we have to show that the above identity
   actually holds out of the diagonal of $\R^n\times\R^n$.
   To this end, let $x,y\in\R^n$ with $x\neq y$ and let
   $r > 0$ be such that $y\notin B_r(x)$.
   Since $\R^n\setminus E$ is dense in $\R^n$ (as $E$ has measure $0$), there exists
   a sequence $\{x_j\}_j$ in $(\R^n\setminus E)\cap B_r(x)$ convergent to $x$ as $j\to\infty$;
   hence, identity \eqref{app:eqGammaSymQuasi} implies that
   $\Gamma(x_j;y) = \Gamma(y;x_j)$, for every $j\in\mathbb{N}$.
   From this, as $\Gamma$ is continuous out of the diagonal of
   $\R^n\times\R^n$, we  deduce that $\Gamma(x;y) = \Gamma(y;x)$.
  \end{proof}
  \begin{cor}\label{app:cor_SmoothGammaBis}
   If $x\in\RN$, then $\Gamma(\cdot;x)\in C^\infty(\R^n\setminus\{x\},\R)$, and
   $\Lop \Gamma(\cdot; x) = 0$ outside $x$.
  \end{cor}
  \begin{thm} \label{app:thm_smoothGammaxy}
   The function $\Gamma$ is smooth out of the diagonal of $\RN\times\RN$.
  \end{thm}
  \begin{proof}
   We introduce the $2m$ vector fields $\widetilde{X}_1(x,y),\ldots,\widetilde{X}_m(x,y),\widetilde{Y}_1(x,y),\ldots,\widetilde{Y}_m(x,y)$,
  operating on $(x,y)\in \RN\times\RN$, defined as follows
  $$\widetilde{X}_i(x,y) = X_i(x), \quad \widetilde{Y}_i(x,y) = X_i(y)\qquad \text{for every $i = 1,\ldots,m $.}$$
    We then set $\widetilde{\Lop} := \sum_{j = 1}^m(\widetilde{X}_j^2+\widetilde{Y}_j^2)$.
   Obviously, $\widetilde{\Lop}$ has smooth coefficients; moreover,
   since $[\widetilde{X}_i,\widetilde{Y}_j] = 0$ for every $i,j = 1,\ldots, m$, it is immediate to see that
   $\widetilde{\Lop}$ is a H\"ormander operator on $\RN\times\RN$, hence $C^\infty$-hypoelliptic.
   On the other hand, since Corollaries \ref{sec.two_2:corSmoothGamma}
   and \ref{app:cor_SmoothGammaBis} imply that
   $$\widetilde{\Lop}\Gamma(x;y) = \Lop \Gamma(\cdot;y)
   + \Lop \Gamma(x;\cdot) = 0, \quad \text{for every
   $x,y\in\RN$ with $x\neq y$},$$
   and since $\Gamma$ is continuous out of the diagonal of $\RN\times\RN$
   (see Proposition \ref{sec.two_2:propWellGammaBis}),
   we infer that
   $\Gamma$ is smooth on the same set.
  \end{proof}

 

\begin{thebibliography}{100}
\bibitem{Aarao}
  J. Aar$\tilde{\text{a}}$o:
   \emph{A transport equation of mixed type}, Dissertation, Yale University
  1997.

  \bibitem{Aarao2}
  J. Aar$\tilde{\text{a}}$o:
   \emph{Fundamental solutions for some partial differential operators from fluid dynamics and statistical physics},
 SIAM Rev.  {49} (2007), 303--314.


 \bibitem{AbbondanzaBonf}
 B. Abbondanza, A. Bonfiglioli:
\emph{On the Dirichlet problem and the inverse mean value theorem
 for a class of divergence form operators}, J. London Math. Soc.,  {87} (2013), 321--346.

   \bibitem{AgrachevBoscainGauthier}
  A. Agrachev, U. Boscain, J.-P. Gauthier, F. Rossi:
 \emph{The intrinsic hypoelliptic Laplacian and its heat kernel on unimodular Lie groups},
 J. Funct. Anal.  {256} (2009), 2621--2655.


 \bibitem{BattagliaBonf}
  E. Battaglia, A. Bonfiglioli:
 \emph{Normal families of functions for subelliptic operators
  and the theorems of Montel and Koebe},  J. Math. Anal. Appl.,  {409} (2014), 1--12.


 \bibitem{BauerFurutaniIwasaki}
 W. Bauer, K. Furutani, C. Iwasaki:
 \emph{Fundamental solution of a higher step Grushin type operator},
 Adv. Math.  {271} (2015), 188--234.

  \bibitem{Beals}
  R. Beals:
  \emph{Exact fundamental solutions},
    ``\'Equations aux D\'eriv\'ees Partielles'' (Saint-Jean-de-Monts, 1998),
    Exp. No. I, 9 pp., Univ. Nantes, Nantes, 1998.

    \bibitem{Beals2}
     R. Beals:
     \emph{A note on fundamental solutions},
     Comm. Partial Differential Equations  {24} (1999), no. 1--2, 369--376.


    \bibitem{BealsGaveauGreiner1}
  R. Beals, B. Gaveau, P. Greiner:
  \emph{On a geometric formula for the fundamental solution of subelliptic Laplacians},
  Math. Nachr.  {181} (1996), 81--163.

     \bibitem{BealsGaveauGreiner2}
  R. Beals, B. Gaveau, P. Greiner:
 \emph{The Green function of model step two hypoelliptic operators and the analysis of certain tangential Cauchy Riemann complexes}, Adv. Math.  {121} (1996), no. 2, 288--345.



 \bibitem{BealsGaveauGreiner2ter}
   R. Beals, B. Gaveau, P. Greiner:
 \emph{Complex Hamiltonian mechanics and parametrices for subelliptic Laplacians. I. II. III.},
 Bull. Sci. Math.  {121} (1997), no. 1-2-3, 1--36; 97--149; 195--259.

 \bibitem{BealsGaveauGreiner3}
  R. Beals, B. Gaveau, P. Greiner:
      \emph{Uniform hypoelliptic fundamental solutions},
      J. Math. Pures Appl. (9)  {77} (1998), no. 3, 209--248.

     \bibitem{BealsGaveauGreiner4}
  R. Beals, B. Gaveau, P. Greiner:
      \emph{Hamilton-Jacobi theory and the heat kernel on Heisenberg groups},
      J. Math. Pures Appl. (9)  {79} (2000), no. 7, 633--689.



        \bibitem{BealsGaveauGreinerKannai}
  R. Beals, B. Gaveau, P. Greiner, Y. Kannai:
 \emph{Exact fundamental solutions for a class of degenerate elliptic operators},
 Comm. Partial Differential Equations  {24} (1999), no. 3--4, 719--742.

         \bibitem{BealsGaveauGreinerKannai2}
  R. Beals, B. Gaveau, P. Greiner, Y. Kannai:
 \emph{Transversally elliptic operators},
 Bull. Sci. Math.  {128} (2004), no. 7, 531--576.


        \bibitem{BealsGreinerGaveau}
  R. Beals, P. Greiner, B. Gaveau,:
        \emph{fundamental solutions for some highly degenerate elliptic operators},
        J. Funct. Anal.  {165} (1999), no. 2, 407--429.

 \bibitem{BensonDooleyRatcliff}
 C. Benson, A.H. Dooley, G. Ratcliff:
\emph{Fundamental solutions for powers of the Heisenberg sub-Laplacian},
 Illinois J. Math.  {37} (1993), 455--476.

 \bibitem{Bieske}
 T. Bieske:
 \emph{Fundamental solutions to the $p$-Laplace equation in a class of Grushin vector fields},
 Electron. J. Differential Equations  {2011} (2011), No. 84, 1--10.

 \bibitem{BieskeGong}
  T. Bieske, J. Gong:
  \emph{The $p$-Laplace equation on a class of Grushin-type spaces},
 Proc. Amer. Math. Soc.  {134} (2006), 3585--3594.



 \bibitem{BoggessRaich}
 A. Boggess, A. Raich:
 \emph{A simplified calculation for the fundamental solution to the heat equation on the Heisenberg group},
 Proc. Amer. Math. Soc.  {137} (2009), 937--944.

 \bibitem{BonfLancCPAA}
  A. Bonfiglioli, E. Lanconelli:
  \emph{Lie groups related to H\"ormander operators and Kolmogorov-Fokker-Planck equations}
 Commun. Pure Appl. Anal.  {11} (2012), 1587--1614.

 \bibitem{BonfLancJEMS}
  A. Bonfiglioli, E. Lanconelli:
  \emph{Subharmonic functions in sub-Riemannian settings},
 J. Eur. Math. Soc.  {15}  (2013), 387--441.

 \bibitem{BonfLancTomm}
  A. Bonfiglioli, E. Lanconelli, A. Tommasoli:
 \emph{Convexity of average operators
 for subsolutions to subelliptic equations},
 Analysis \& PDE,  {7} (2014), 345--373.

\bibitem{BLUade}
 A. Bonfiglioli, E. Lanconelli, F. Uguzzoni:
 \emph{Uniform Gaussian estimates of the fundamental
 solutions for heat operators on Carnot groups}
 Adv. Differential Equations  {7}  (2002), 1153--1192.

\bibitem{BLUTRANS}
 A. Bonfiglioli, E. Lanconelli, F. Uguzzoni:
 \emph{Fundamental solutions for non-divergence
 form operators on stratified groups}
 Trans. Amer. Math. Soc.,
  {356}  (2004), 2709--2737.

\bibitem{BLU}
 A. Bonfiglioli, E. Lanconelli, F. Uguzzoni:
 \emph{Stratified {L}ie groups and potential theory for their
  sub-{L}aplacians}, Springer Monographs in Mathematics, Springer: Berlin, 2007.

\bibitem{Bony}
 J.-M. Bony:
 \emph{Principe du maximum, in{\'e}galit{\'e} de {H}arnack et unicit{\'e}
 du probl{\`e}me de {C}auchy pour les op{\'e}rateurs elliptiques
 d{\'e}g{\'e}n{\'e}r{\'e}s},
 Ann. Inst. Fourier (Grenoble),
  {19} (1969),
 277--304.


 \bibitem{BoscainGauthierRossi}
 U. Boscain, J.-P. Gauthier, F. Rossi:
 \emph{The hypoelliptic heat kernel over three-step nilpotent Lie groups},
 Sovrem. Mat. Fundam. Napravl.  {42} (2011), 48--61; transl. in:
 J. Math. Sci. (N. Y.)  {199} (2014), no. 6, 614--628.

 \bibitem{BramBranLU}
 M. Bramanti, L. Brandolini, E. Lanconelli,  F. Uguzzoni:
  \emph{Non-divergence equations structured on H\"{o}rmander vector fields: heat kernels and Harnack inequalities},
 Mem. Am. Math. Soc.  {201} (2010), 1--123.

 \bibitem{BramBranMP}
 M. Bramanti, L. Brandolini, M. Manfredini, M. Pedroni:
 \emph{Fundamental solutions and local solvability for non-
smooth H\"ormander's operators}, to appear in Mem. Am. Math. Soc., arXiv: 1305.3398.


 \bibitem{CalinChang}
 O. Calin, D.-C. Chang:
 \emph{Heat kernels for differential operators with radical function coefficients},
 Taiwanese J. Math.  {15} (2011), 1629--1636.


 \bibitem{CalinChangFurutaniIwasaki}
 O. Calin, D.-C. Chang, K. Furutani, C. Iwasaki:
 \emph{Heat kernels for elliptic and sub-elliptic operators},
 Methods and techniques. Applied and Numerical Harmonic Analysis. Birkh\"auser/Springer, New York, 2011.


 \bibitem{CapognaDanielliGarofalo}
 L. Capogna, D. Danielli, N. Garofalo:
  \emph{Capacitary estimates and subelliptic
 equations}, Amer. J. Math.  {118} (1996), 1153--1196.

 \bibitem{Chandrasekhar}
    S. Chandrasekhar:
    \emph{Stochastic problems in physics and astronomy},
  Rev. Mod. Phys.  {15} (1943), 1--89.





\bibitem{CittiManfredini}
 G. Citti, M. Manfredini:
 \emph{Uniform estimates of the fundamental solution
 for a family of hypoelliptic operators},
 Potential Anal.  {25} (2006) 147--164.

 \bibitem{Cygan}
 J. Cygan:
 \emph{Heat kernels for class $2$ nilpotent groups},
 Studia Math.  {64} (1979) 227--238.

\bibitem{Dixmier}
 J. Dixmier:
  \emph{Sur les repr\'esentations unitaires des groupes de Lie nilpotents. I},
  Amer. J. Math.  {81} (1959) 160--170.


 \bibitem{Folland}
 G.B. Folland:
  \emph{A fundamental solution for a subelliptic operator},
  Bull. Amer. Math. Soc.  {79} (1973), 373--376.

 \bibitem{Folland75}
 G.B. Folland:
  \emph{Subelliptic estimates and function spaces on nilpotent Lie groups},
  Ark. MAt.  {13} (1975), 161--207.

   \bibitem{Folland2}
 G.B. Folland:
  \emph{On the Rothschild-Stein lifting theorem},
 Comm. Partial Differential Equations  {2} (1977), 161--207.

 \bibitem{FollandStein}
 G.B. Folland,  E.M. Stein:
  \emph{Estimates for the $\overline{\partial}_b$-complex and analysis on
 the Heisenberg group}, Comm. Pure Appl. Math.  {27} (1974), 429--522.


\bibitem{Friedman} 	
 A. Friedman:
 \emph{Partial differential equations of parabolic type},
 Englewood Cliffs, N.J.: Prentice-Hall, 1964.


 \bibitem{Furutani}
 K. Furutani:
 \emph{Heat kernels of the sub-Laplacian and the Laplacian on nilpotent Lie groups},
 In: Analysis, geometry and topology of elliptic operators, 173--214, World Sci. Publ., Hackensack, NJ, 2006.


 \bibitem{Gaveau}
 B. Gaveau:
  \emph{Principe de moindre action, propagation de la chaleur et estim\'ees
sous-elliptiques sur certains groupes nilpotents}, Acta Math.  {139} (1977), 95--153.

 \bibitem{Greiner}
 P.C. Greiner:
  \emph{A fundamental solution for a nonelliptic partial differential
operator}, Can. J. Math.  {31} (1979), 1107--1120.

 \bibitem{HeinonenHolopainen}
 J. Heinonen, I. Holopainen:
  \emph{Quasiregular maps on Carnot groups},
  J. Geom. Anal.  {7} (1997), 109--148.

\bibitem{HormACTA}
 {{L. H\"ormander}}:
 \emph{Hypoelliptic second order
 differential equations}, Acta Math.  {119} (1967), 147--171.

 \bibitem{HormBOOK}
 L. H\"ormander:
 \emph{The analysis of linear partial differential operators, 1;4}, Springer (1983;1985), ch. 7;18.

 \bibitem{Hulanicki}
 A. Hulanicki:
  \emph{The distribution of energy in the Brownian motion in the
Gaussian field and ana\-ly\-tic-hypo\-ellip\-ti\-city of certain subelliptic operators on
the Heisenberg group}, Studia Math.  {56} (1976), 165--173.

 \bibitem{Kaplan}
  A. Kaplan:
 \emph{Fundamental  solutions  for  a  class  of  hypoelliptic  PDE  generated  by  composition  of
quadratic forms}, Trans. Amer. Math. Soc.  {258} (1980), 147--153.

 \bibitem{Klinger}
 A. Klingler:
  \emph{New derivation of the Heisenberg kernel}, Comm. Partial Differ. Equations  {22} (1997),
2051--2060.


    \bibitem{Kolmogorov}
    A.N. Kolmogorov: \emph{Zuf\"allige Bewegungen}, Acta Math.  {35} (1934), 116--117.

 \bibitem{Levi}
 E.E. Levi:
  \emph{Sulle equazioni lineari totalmente ellittiche alle derivate parziali},
  Rend. Circ. Mat. Palermo  {24} (1907), 275--317.



    \bibitem{LuanZhu}
    J. Luan, F. Zhu:
    \emph{The heat kernel on the Cayley
    Heisenberg group},
 Acta Math. Sci. Ser. B Engl. Ed.  {25} (2005), 687--702.

       \bibitem{LuanZhu2}
    J. Luan, F. Zhu:
    \emph{The Green functions on Cayley Heisenberg groups},
 J. Math. (Wuhan)  {29} (2009), no. 4, 395--400.

\bibitem{NagelRicciStein}
 A. Nagel, F. Ricci, E.M. Stein:
 \emph{Fundamental solutions and harmonic analysis on nilpotent
 groups},
 Bull. Am. Math. Soc., New Ser.,  {23}  (1990),  139--144.


\bibitem{RothschildStein}
  L.P. Rothschild, E.M. Stein:
 \emph{Hypoelliptic differential operators and nilpotent groups},
 Acta Math.  {137} (1976), 247--320.

\bibitem{Treves}
 {F. Treves}:
 \emph{Topological vector spaces, distributions and kernels},
 Academic Press, London, 1967.

    \bibitem{Tsutsumi}
    C. Tsutsumi:
 \emph{The fundamental solution for a parabolic pseudo-differential operator and parametrices for degenerate operators },
 Proc. Japan Acad.  {51} (1975), 103--108.


 \bibitem{Zhu}
 F. Zhu:
 \emph{The heat kernel and the Riesz transforms on the quaternionic Heisenberg groups},
 Pacific J. Math.  {209} (2003), no. 1, 175--199.

  \bibitem{ZhuYang}
 F. Zhu, Q. Yang:
  \emph{Heat kernels and Hardy's uncertainty principle on H-type groups},
 Acta Math. Sci. Ser. B Engl. Ed.  {28} (2008), 171--178.
\end{thebibliography}
 \end{document}